\documentclass[12pt]{amsart}
\usepackage{amsmath,amssymb,latexsym, amsthm, amscd, mathrsfs, stmaryrd} 
\usepackage[all]{xy}

\usepackage[all]{xy}
\usepackage{color}
\usepackage{hyperref}

\theoremstyle{definition}
\newtheorem{rem}[subsubsection]{Remark}
\theoremstyle{plain}
\newtheorem{prop}[subsubsection]{Proposition}
\newtheorem{thm}[subsubsection]{Theorem}
\newtheorem{lem}[subsubsection]{Lemma}
\newtheorem{cor}[subsubsection]{Corollary}

\newcommand{\mbb}{\mathbb}
\newcommand{\mbf}{\mathbf}
\newcommand{\mrm}{\mathrm}
\newcommand{\bj}{\mbf j'}

\newcommand{\CC}{\mathscr C}

\newcommand{\e}{\mbf e}
\newcommand{\f}{\mbf f}
\newcommand{\G}{\mbf G}
\newcommand{\h}{\mbf h}
\renewcommand{\H}{\mbf H}
\newcommand{\IC}{\mrm{IC}}
\newcommand{\F}{\mathcal F}
\newcommand{\Gr}{\mrm{Gr}}

\newcommand{\M}{\mathcal M}
\newcommand{\N}{\mathcal N}
\renewcommand{\P}{\mbf P}
\newcommand{\Q}{\mathcal Q}

\renewcommand{\t}{\mbf t}
\newcommand{\U}{\mathrm U}

\newcommand{\End}{\mrm{End}}
\newcommand{\Ext}{\mrm{Ext}}
\newcommand{\Hom}{\mrm{Hom}}

\newcommand{\ro}{\mrm{ro}}
\newcommand{\co}{\mrm{co}}

\newcommand{\bU}{\mbf H}

\newcommand{\ve}{\varepsilon}

\title[]{Quasi-split symmetric pairs of  $\U(\mathfrak{sl}_n)$
\\ and \\ Steinberg varieties of classical type}

\author{Yiqiang Li}
\address{University at Buffalo\\ the State University of New York}
\email{yiqiang@buffalo.edu}


\begin{document}

\begin{abstract}
We provide a Lagrangian construction for  the fixed-point subalgebra, together with its idempotent form, in a quasi-split symmetric pair of type $A_{n-1}$.
This is obtained inside the limit of a projective system of Borel-Moore homologies of the Steinberg varieties of $n$-step isotropic flag varieties.
Arising from the construction are a basis of homological origin for the idempotent form and
a geometric realization of rational modules.
\end{abstract}

\maketitle


\section*{Introduction}

The geometric study of representation theory of  Kac-Moody algebras  via Nakajima quiver varieties
has been one of the actively-pursued themes  in recent years.
A key result in this theme is a geometric realization of the universal enveloping algebra of a symmetric Kac-Moody algebra (or rather its Schur quotient)~\cite{N98}.
Remarkably, all irreducible integrable highest weight modules  are naturally realized as well from the construction.

A precursor of the above key result is Ginzburg's construction in the $\mathfrak{sl}_n$ case, which makes use of  the Borel-Moore homology of the Steinberg variety of the $n$-step flag variety of a general linear group (\cite{G91, CG97}).
Ginzburg's work in turn is inspired by the  construction of quantum $\mathfrak{gl}_n$ via perverse sheaves on an $n$-step flag variety by Beilinson-Lusztig-MacPherson~\cite{BLM}. 

As a classical analogue of ~\cite{BLM}, it is observed in~\cite{BKLW} that  the convolution algebra of perverse sheaves on  an $n$-step isotropic flag variety  
provides a realization of  the Schur quotient of a quantum version of the fixed-point subalgebra in a quasi-split symmetric pair of $\mathfrak{sl}_n$.
Shortly afterwards, it is further observed by the author that the Schur quotient of the subalgebra itself can be realized as Borel-Moore homology of the Steinberg variety of an $n$-step isotropic flag variety. 
Based on these facts,  an extension of  the above theme was proposed in~\cite{Li19a}. 
On the algebra side, Kac-Moody algebras  are replaced by their   fixed-point subalgebras  under an involution. 
On the geometry side,
 Nakajima varieties are replaced by a class of twisted quiver varieties, called  $\sigma$-quiver varieties, constructed therein as fixed-point loci of Nakajima varieties under certain symplectic involutions. 
 In particular, cotangent bundles of $n$-step isotropic flag varieties are examples of $\sigma$-quiver varieties. 
Since a Kac-Moody algebra together with its fixed-point forms a  symmetric pair,  this extended theme can be thought of as a geometric representation theory of symmetric pairs. 

The purpose of this article is to provide details for the above-mentioned link between the two subjects in the title. Indeed, we achieve more. That is, 
we  provide a construction of the universal enveloping algebra, together with its idempotent form,
of the subalgebra in a quasi-split symmetric pair of $\mathfrak{sl}_n$. This is obtained
inside the limit of a projective system of  Borel-Moore homologies of  the Steinberg varieties of $n$-step isotropic flag varieties. 
Moreover, we obtain a basis of homological origin for the idempotent form and we show that
modules arising from this construction are rational. 
Overall, this work is a classical analogue of Ginzburg's, and hence in some sense the Langlands dual of~\cite{BKLW, LW18}. It sits in between the works~\cite{BKLW, Li19a}  and thus is likely to shed new light on the extended theme  in~\cite{Li19a}, especially on the main conjecture therein which draws a link between a quasi-split symmetric pair of type $ADE$ and a $\sigma$-quiver variety. 

Just as in~\cite{BLM, CG97}, the above projective limit is  a substitute for the non-existent Borel-Moore homology of 
Steinberg varieties of $n$-step isotropic flag of the classical ind-groups $\mrm O_\infty$ and $\mrm{Sp}_\infty$.

As a byproduct, we  observe that Spaltenstein varieties of classical groups are isomorphic under Kraft-Procesi's row reduction and we conjecture 
it be true under the column reduction as well. 
This reduces the characterization of the Lagrangian property of Spaltenstein varieties of classical type to the minimal degeneration cases (see~\cite{Li19b, KP82}). 

Our analysis goes in parallel with, but not a copycat of,  that of Ginzburg. There are a few completely non-trivial modifications. 
The first one is the verification of the rank-one nonhomogeneous Serre relations in the defining relations of the fixed point subalgebra, say $\mathfrak{sl}_n^{\theta}$.
Instead of proving it directly, which seems impossible by using only the machinery of Borel-Moore homology, we build up enough structures so that it falls out naturally. 
This treatment is  borrowed from~\cite{FLW}, and can be used to simplify the argument for $\mathfrak{sl}_2$-case  in~\cite{CG97}. 
The second one is the determination of the Chevalley generators in the $n$-being-even case. 
The naively-defined one is simply false  and the correct definition is secured by  using  the fact that 
the algebra $\mathfrak{sl}_n^{\theta}$ is a subalgebra of $\mathfrak{sl}_{n+1}^{\theta}$.
This treatment has its root in~\cite{BKLW}.
The third one is that not all finite-dimensional simple modules of $\mathfrak{sl}_n^{\theta}$ can be realized geometrically in our setting.
Instead, we show that only those rational ones appear and all rational modules appear this way. 
The fourth one is that the even orthogonal case requires extra care because the associated group is disconnected, but the difficulty disappears by bringing in a $\mbb Z/2\mbb Z$-equivariant condition. A transfer map from odd orthogonal groups to even ones of smaller ranks also simplifies the analysis in the even case. 

At some point in the stabilization process, we have to resort to the normality of certain classical nilpotent orbits.
Mysteriously, these orbits showed up as well in the study of $W$-algebras (\cite{AM18}) which suggests intricate relations between 
symmetric pairs and $W$-algebras deserving to be further explored. 

\subsection*{Acknowledgements}
We thank  Prof. Ginzburg for helpful explanations on his book~\cite{CG97}. 
We thank Xuhua He and Dan Nakano for interesting discussions. 
Main results in this paper were announced in the AMS special session ``Geometric Methods in Representation Theory,'' UC Riverside, November 9-10, 2019.
We thank the organizers for the invitation. 
This work was partially supported by the NSF  grant DMS-1801915.
We thanks the anonymous referee for a careful  proofreading and many useful suggestions.

\tableofcontents

\section{Quasi-split symmetric pair of type $A_{n-1}$}

We recall the quasi-split symmetric pairs $(\mathfrak{sl}_n,\mathfrak{sl}^{\theta}_n)$ of $\mathfrak{sl}_n$  and their idempotent forms in this section. 
We also introduce   rational  $\mathfrak{sl}^{\theta}_n$-modules, which arise naturally from geometry. 

\subsection{Presentation}
\label{Pre}

Let $n$ be an integer greater than $1$, and we fix forever 
\[
r\equiv r(n) =\lfloor n/2\rfloor.
\]
Let $\mathfrak{sl}_n=\mathfrak{sl}_n (\mbb Q)$ be the special linear Lie algebra over $\mbb Q$ of rank $n-1$ with standard Chevalley generators
$\{ e_i, f_i, h_i| 1\leq i\leq n-1\}$. Let
$\theta$ be the involution on the set $\{ 1,\cdots,n-1\}$ defined by
$\theta(i) = n-i $ for any  $1\leq i\leq n-1.$
The involution $\theta$ induces a Lie algebra  involution on $\mathfrak{sl}_n$, denoted by the same notation, by
\[
e_i\mapsto f_{\theta(i)}, f_i\mapsto e_{\theta(i)}, h_i\mapsto - h_{\theta(i)}, 1\leq i\leq n-1.
\]
Let $\mathfrak{sl}_n^{\theta}$ be the fixed-point subalgebra of $\mathfrak{sl}_n$ under $\theta$, which is known to be  reductive.
The pair  $(\mathfrak{sl}_n,\mathfrak{sl}_n^{\theta})$  is a quasi-split symmetric pair and 
further the pair $(\mathfrak{sl}_2,\mathfrak{sl}_2^{\theta})$ is  split. 
One can define in a similar way an involution $\theta$ on $\mathfrak{gl}_n$ and then
$
\mathfrak{sl}^{\theta}_n  = \mathfrak{gl}_n^{\theta} \cap \mathfrak{sl}_n$ and $\mathfrak{gl}_n^{\theta} = \mathfrak{sl}_n^{\theta}\oplus \mbb Q$
with $\mathfrak{gl}^{\theta}_n \cong   \mathfrak{gl}_{\lfloor \frac{n}{2}\rfloor} \oplus \mathfrak{gl}_{\lceil \frac{n}{2}\rceil} $.
Note that $\mathfrak{sl}^{\theta}_n \cong   \mathfrak{sl}_{\lfloor \frac{n}{2}\rfloor} \oplus \mathfrak{gl}_{\lceil \frac{n}{2}\rceil} $.
The algebra $\mathfrak{sl}_n^{\theta}$ is generated by the elements
\[
e_{i,\theta} = e_i+f_{\theta(i)}, h_{i,\theta}= h_i-h_{\theta(i)},\quad \forall 1\leq i\leq n-1.
\] 
For convenience, we set
$f_{i,\theta} = f_i+ e_{\theta(i)}$ and clearly  $e_{i,\theta} = f_{\theta(i),\theta}$. 

Let $\tau: \mathfrak{sl}_n\to \mathfrak{sl}_n$ be the Chevalley involution defined by $e_i\mapsto f_i, f_i\mapsto e_i$, and $h_i\mapsto -h_i$ for all $1\leq i\leq n-1$. Clearly, the algebra $\mathfrak{sl}_n^{\theta}$ is stable under $\tau$ and
hence $\tau$ restricts to an involution, still denoted by $\tau$,  on $\mathfrak{sl}_n^{\theta}$ defined by
\begin{align}
\label{tau}
\tau: e_{i,\theta} \mapsto f_{i,\theta}, h_{i,\theta} \mapsto -h_{i,\theta} \quad \forall 1\leq i\leq n-1.
\end{align}
Let $\U(\mathfrak{sl}_n^{\theta})$ be the universal enveloping algebra of $\mathfrak{sl}_n^{\theta}$ over $\mbb Q$.
Let $(c_{ij})_{1\leq i,j\leq n-1}$, $c_{ij}=2\delta_{i,j}-\delta_{i, j+1}-\delta_{i,j-1}$, be the Cartan matrix of $\mathfrak{sl}_n$.
Thanks to the fact that $\mathfrak{gl}_n^{\theta}=\mathfrak{sl}^{\theta}_n \oplus \mbb C$ and~\cite{LZ19}, we know that the algebra $\U(\mathfrak{sl}^{\theta}_n)$ admits a presentation
by generators and relations depending on the parity of $n$ as follows.
\begin{align}
\label{Def-odd}
\begin{cases}
h_{i,\theta} + h_{\theta(i),\theta}  =0, [h_{i,\theta}, h_{j, \theta}]=0  & \forall 1\leq i,j\leq n-1. \\
[h_{i, \theta}, e_{j,\theta}]  = (c_{ij} - c_{\theta(i), j}) e_{j,\theta}  & \forall 1\leq i, j\leq n-1.\\
[e_{i,\theta}, e_{j,\theta}]  = \delta_{i, \theta(j)} h_{i,\theta}  &  \mbox{if} \ c_{ij}=0.  \hspace{3cm} ($n$ \ \mbox{odd})\\
e_{i,\theta}^2  e_{j,\theta} - 2 e_{i,\theta} e_{j,\theta} e_{i,\theta} + e_{j,\theta} e_{i,\theta}^2 = 0 
& \mbox{if} \ c_{ij}=-1, i\neq  \theta(j). \\
e_{i,\theta}^2  e_{j,\theta} - 2 e_{i,\theta} e_{j,\theta} e_{i,\theta} + e_{j,\theta} e_{i,\theta}^2 =  - 4  e_{i,\theta}
& \mbox{if} \ c_{ij} =-1, i = \theta(j).
\end{cases}
\end{align}
\begin{align}
\label{Def-even}
\begin{cases}
h_{i, \theta} + h_{\theta(i), \theta}  =0, [h_{i,\theta},h_{j,\theta}]=0  & \forall 1\leq i, j\leq n-1. \\
[h_{i,\theta}, e_{j,\theta}]  = (c_{ij} - c_{\theta(i), j}) e_{j,\theta}  & \forall 1\leq i, j\leq n-1.\\
[e_{i,\theta}, e_{j,\theta}]  = \delta_{i, \theta(j)} h_{i,\theta}  &  \mbox{if} \ c_{ij}=0. \hspace{3cm}  ($n$ \ \mbox{even}) \\
e_{i,\theta}^2  e_{j,\theta} - 2 e_{i,\theta} e_{j,\theta} e_{i,\theta} + e_{j,\theta} e_{i,\theta}^2  = 0 
& \mbox{if} \ c_{ij}=-1, i\neq \theta(i). \\
e_{i,\theta}^2  e_{j,\theta} - 2 e_{i,\theta} e_{j,\theta} e_{i,\theta} + e_{j,\theta} e_{i,\theta}^2   =    e_{j, \theta}
& \mbox{if} \ c_{ij} =-1, i = \theta(i). 
\end{cases}
\end{align}

\begin{rem}
In the notations of~\cite{LZ19}, $h_{i,\theta}= \mbf d_i - \mbf d_{i+1}$,
$e_{i,\theta}= \mbf e_i$ and $f_{i,\theta} = \mbf f_i$.  
\end{rem}

\subsection{Lusztig form}
For each $v\in \mbb N$, we set
\[
\Lambda_{v+\infty} =
\left \{
\lambda =( \lambda_i)_{1\leq i\leq n} \in \mbb N^n| \sum_{i=1}^n \lambda_i \equiv v - \delta_{-1, (-1)^v} (\mbox{mod}\ 2n), \lambda_i =\lambda_{n+1-i}
\right \}.
\]
The set $\Lambda_{v+\infty}$ admits a partition
\begin{align}
\label{Lambda-v}
\Lambda_{v+\infty} = \sqcup_{k: v+2kn>0} \Lambda_{v+2kn}, \Lambda_{v+2kn} =\{ \lambda \in \Lambda_{v+\infty}| \sum_i \lambda_i= v+2kn-\delta_{-1,(-1)^v}\}.
\end{align}
We define an equivalence relation $\sim$ on $\Lambda_{v+\infty}$ by
\[
\lambda \sim \mu \ \mbox{if and only if} \ \lambda -\mu \equiv 0 \ (\mbox{mod} \ (2, \cdots, 2)).
\]
Let $\bar \Lambda_{v+\infty} = \Lambda_{v+\infty}/\sim$ and $\bar \lambda$ be the equivalence class of $\lambda$. 

Let $(\delta_i)_{1\leq i\leq n}$ be the standard basis of $\mbb N^n$. 

Let $\dot\U(\mathfrak{sl}^{\theta}_n)$ be the idempotent form of $\U(\mathfrak{sl}^{\theta}_n)$, similar to Lusztig's idempotent form of $\U(\mathfrak{sl}_n)$. 
This is an associative  $\mbb Q$-algebra, without unit,  with generators 
$1_{\bar\lambda}$ and  $e_{i,\theta} 1_{\bar\lambda}$, $\forall \bar \lambda \in \bar\Lambda_{v+\infty}$, with  $v \in I_{n,\ve}$
where
\begin{align}
\label{I}
I_{n, \ve} =
\begin{cases}
\{ 2 \ell | 1\leq \ell \leq   n\} & \mbox{if} \ \ve =-1\ \mbox{or}\  \ve=1, n \ \mbox{even},\\
\{ 2\ell -1 | 1\leq \ell \leq n\} & \mbox{if} \ \ve =1, n \ \mbox{odd}.
\end{cases}
\end{align} 
and as $\U(\mathfrak{sl}_n^{\theta})$-bimodule generated by $1_{\bar \lambda}$ and such that
\begin{align*}
1_{\bar \lambda} 1_{\bar \mu}= \delta_{\bar \lambda, \bar \mu} 1_{\bar \lambda},  \quad \forall \bar \lambda,\bar \mu,\\
 1_{\bar\lambda} h_{i,\theta} =
h_{i,\theta} 1_{\bar\lambda} = (\lambda_{i+1} -\lambda_i)1_{\bar \lambda},\quad \forall i, \bar\lambda,\\
e_{i,\theta} 1_{\bar\lambda}  =  1_{\overline{\lambda +  \delta_i -  \delta_{i+1} -  \delta_{\theta(i) } +  \delta_{\theta(i) + 1} }} e_{i,\theta},\  \forall i, \bar \lambda.
\end{align*}

\subsection{Rational  modules}
\label{Rep}

We define $h'_{i,\theta}$ for $1\leq i\leq r$ inductively as follows. 
\[
h'_{r, \theta} =
\begin{cases}
[e_{r, \theta}, f_{r,\theta}], &\mbox{if $n$ is odd},\\
e_{r,\theta}, &\mbox{if $n$ is even},
\end{cases}
h'_{i, \theta} = [[e_{i,\theta}, h'_{i+1,\theta}],f_{i,\theta}], \quad 1\leq i \leq r-1.
\]
Let $\mathfrak h_{\theta}$ be the Lie subalgebra generated by $h_{i,\theta}$ and $h'_{i,\theta}$ for $1\leq i \leq r$.
It is known from~\cite{LZ19} that $\mathfrak h_{\theta}$ is abelian with basis $\{h_{i,\theta}, h'_{i,\theta}| 1\leq i\leq r\}$ if $n$ odd and $\{h_{i, \theta}, h'_{i', \theta}| 1\leq i\leq r-1, 1\leq i'\leq r\}$ if $n$ is even. 
Let $\U(\mathfrak h_{\theta})$ be the universal enveloping algebra of $\mathfrak h_{\theta}$. 
The algebra $\U(\mathfrak h_{\theta})$ is a Cartan subalgebra of $\U(\mathfrak{sl}_n^{\theta})$.
For any pair $(\omega, \omega')$, where $\omega=(\omega_i), \omega'=(\omega'_i)\in \mbb Z^r$, 
we define the one dimensional $\U(\mathfrak h_{\theta})$-module $\mbb Q_{\omega, \omega'}$  by
\[
h_{i,\theta} . x= \omega_i x, h'_{i, \theta} . x = \omega'_i  x, \quad \forall x\in \mbb Q. 
\]
When $n$ is even, we must have $\omega_r=0$ since $h_{r,\theta}=0$. 
Let $\U^+$ be the subalgebra of $\U(\mathfrak{sl}_n^{\theta})$ generated by $h_{i, \theta}$, $h'_{i,\theta}$ and $e_{i,\theta}$ for $1\leq i \leq r$. 
The module $\mbb Q_{\omega,\omega'}$ is naturally a $\U^+$-module induced via the projection $\U^{+}\to \U(\mathfrak h_{\theta})$. 
Let 
\[
M(\omega,\omega') = \U(\mathfrak{sl}_n^{\theta})\otimes_{\U^+} \mbb Q_{\omega, \omega'}
\]
be the Verma module attached to the pair $(\omega,\omega')$. 
Let $\mathcal I_{\omega,\omega'}$ be the unique maximal submodule in $M(\omega, \omega')$ and let 
\begin{align}
\label{simple}
L'(\omega, \omega') =M(\omega, \omega')/\mathcal I_{\omega,\omega'}
\end{align}
be the simple quotient. 
Note that $L'(\omega, \omega') $ can be infinite dimensional. 
Thanks to~\cite{LZ19}, we know that $\U^+$ is a Borel subalgebra in $\U(\mathfrak{sl}_n^{\theta})$.
In particular, we know that, via the inclusion 
$\mathfrak{sl}_n^{\theta}\hookrightarrow \mathfrak{gl}_n^{\theta} \cong \mathfrak{gl}_{\lfloor \frac{n}{2}\rfloor} \oplus \mathfrak{gl}_{\lceil \frac{n}{2}\rceil}$,  
any finite dimensional $L'(\omega,\omega')$ is isomorphic to  $L_1\otimes L_2$
where $L_1$ and $L_2$ are finite-dimensional rational simple modules of $\mathfrak{gl}_{\lfloor \frac{n}{2}\rfloor}$ and $\mathfrak{gl}_{\lceil \frac{n}{2}\rceil}$ respectively.
In this respect, we call $L'(\omega,\omega')$ a rational module.
A finite-dimensional $\U(\mathfrak{sl}_n^{\theta})$-module is rational if it is a direct sum of rational simple modules.  
Let $\CC(\mathfrak{sl}_n^{\theta})$ be the abelian category of rational  $\U(\mathfrak{sl}_n^{\theta})$-modules.

The rational modules are infinitesimal version of rational modules of the group $\mrm{GL}_n^{\theta}$ where $\theta$ is the involution defined by
$g\mapsto JgJ$ for $J$ in (\ref{e}). 

When $n$ is odd, the quantum version of simple modules $L'(\omega, \omega')$ is studied in~\cite{W18}, which 
is denoted by $L'(\lambda, H')$ therein.

\section{Steinberg ind-varieties  $Z_{n, \infty,\ve}$ of classical type}
\label{Z}

In this section, we introduce the  $n$-step Steinberg ind-varieties of classical groups and understand their Borel-Moore homologies, as conceptually the geometric setting for $\U(\mathfrak{sl}_n^{\theta})$, via
the projective system of  Steinberg varieties of classical type together with the transfer maps. 

\subsection{Classical groups and their nilpotent orbits}
\label{classical}
Let us fix an integer $\varepsilon \in \{\pm 1\}$. 
Let $V$ be a finite dimensional complex vector space of dimension $v$. 
We assume that $V$ is equipped with a non-degenerate bilinear form $(-,-)_V$  such that
$(u,u')_V =\ve (u',u)_V$ for all $u,u'\in V$. 
We say that $V$ is equipped with an $\ve$-form.
Let $\G_{v,\ve}$ be the isometry group of $V$ with respect to the form $(-,-)_V$. 
When $\ve=1$, $\G_{v,\ve} = \mathrm O_V $ the orthogonal group of the symmetric form and when $\ve=-1$, $\G_{v,\ve} =\mrm{Sp}_V$ the symplectic group
of the skew-symmetric form. 

Let $\N_{v,\ve}$ be the variety of nilpotent elements in $\mrm{Lie} (\G_{v,\ve})$. Let $\G_{v,\ve}$ act on $\N_{v,\ve}$ by conjugation. 
Let $\P(v)$ be the set of all partitions of $v$.  Let
\[
\P_\ve (v) = 
\left \{
\mu\in \P(v)  | \# \left \{ i | \mu_i \equiv \frac{1 - \ve}{2}  \ \mbox{mod} \ 2\right \} \ \mbox{is even}
\right \},
\]
i.e., $\P_{\ve}(v)$ is  the collection of all partitions such that all even (resp. odd)  parts have even multiplicity when $\ve = 1$ (resp. $\ve= -1$). 
Due to the work of Freudenthal, Gerstenhaber and Hesselink (see~\cite{CM93}), 
it is known that the $\G_{v,\ve}$-orbits in $\N_{v,\ve}$ are parametrized by the set $\P_{\ve}(v)$ by assigning a $\G_{v,\ve}$-orbit to the Jordan type of its elements.
Let $\mathcal O_{\mu, \ve}$ be the $\G_{v,\ve}$-orbit whose elements have Jordan type $\mu$.
Let $\hat \mu$ be the dual partition of $\mu$. 
It is known (see~\cite{CM93}) that
\begin{align}
\label{dim-O}
\dim \mathcal O_{\mu,\ve} = \frac{1}{2} 
\left (
|\mu|^2 - \sum_i \hat \mu_i^2 - \ve \left  (|\mu| - \# \{ j| \mu_j \ \mbox{is odd}\}\right ) 
\right ),
\end{align}
where $|\mu|$ is the sum of all parts in $\mu$.

Define a dominance order $\leq$ on $\P(v)$ by declaring 
\[
\lambda \leq \mu \quad \mbox{if and only if} \quad \sum_{i\leq k} \lambda_i \leq \sum_{i\leq k} \mu_i, \ \forall k. 
\]
Via restriction, we have a partial order on $\P_{\ve}(v)$. 
It is known (see~\cite{CM93}) that 
\begin{align}
\label{O-order}
\mathcal O_{\lambda, \ve} \subseteq \bar {\mathcal O}_{\mu,\ve} \ \mbox{if and only if} \ \lambda \leq \mu. 
\end{align}

\subsection{$n$-Nilcone  $\N_{n, v,\ve}$ and its Slodowy slice $S$}

We define the $n$-nilcone of $\G_{v,\ve}$ to be 
\[
\N_{n, v,\ve} =\{ x\in \N_{v,\ve} | x^n=0\}.
\]
We shall study the irreducibility and normality of $\N_{n, v,\ve}$. 
Let $\G_v=\mrm{GL }(V)$ and $\N_{n, v} $ be its $n$-nilcone. 
To a partition $\mu$, let $\mathcal O_\mu$ be the nilpotent $\G_v$-conjugacy class in $\End(V)$ whose elements have Jordan type $\mu$.
By~\cite[Corollary 4.4.3]{CG97}, $\N_{n, v}=\bar {\mathcal O}_{n^k,\ell}$ where $v=nk+\ell$, $0\leq \ell <n$. 
So
\[
\N_{n, v,\ve} = \bar{\mathcal O}_{n^k, \ell} \cap \N_{v,\ve}.
\]
There is  a largest partition in $\P_\ve(v)$ dominated by $(n^k,\ell)$ (more generally  an arbitrary partition).
So the variety $\N_{n, v,\ve}$ turns out to be  a nilpotent orbit closure again and its partition type is listed as follows.
\begin{align}
\label{N-irr}
\begin{split}
\N_{n, v,\ve} & =
\begin{cases}
\bar{\mathcal O}_{n^{k-1}, (n-1), (\ell+1), \ve} & \mbox{if} \ n, k, \ell, \mbox{are odd},\\
\bar{\mathcal O}_{n^k, \ell, \ve} & \mbox{o.w.}
\end{cases}
\quad \quad \quad \mbox{if} \ \ve=-1.
\\
%
\N_{n, v,\ve} & =
\begin{cases}
\bar{\mathcal O}_{n^k, \ell, \ve} & \mbox{if}  \ n, k\ \mbox{even}\ \ell \ \mbox{odd, or,} \\
& \quad n, \ell \ \mbox{odd}, \ k \ \mbox{even}, \\
\bar{\mathcal O}_{n^{k-1}, (n-1), \ell, 1, \ve} & \mbox{if} \ n \ \mbox{even} \ k, \ell \ \mbox{odd}, \\
\bar{\mathcal O}_{n^k, (\ell-1), 1, \ve} & \mbox{if}\  n, k \ \mbox{odd}, \ \ell \ \mbox{even}, 
\end{cases}
\quad \quad \mbox{if}  \ \ve =1, v\ \mbox{is odd}.
\\
%
\N_{n, v,\ve} & =
\begin{cases}
\bar{\mathcal O}_{n^k, \ell, \ve} & \mbox{if} \ n, k, \ell \ \mbox{odd},\\
\bar{\mathcal O}_{n^k, (\ell-1), 1, \ve} & \mbox{if} \ n \ \mbox{odd}, \ k, \ell \ \mbox{even},\\
\bar{\mathcal O}_{n^{k-1}, (n-1), (\ell+1) ,  \ve} & \mbox{if} \ n, \ell \ \mbox{even} \  k \ \mbox{odd}, \\
\bar{\mathcal O}_{n^k, (\ell-1), 1, \ve} & \mbox{if} \ n, k, \ell \ \mbox{even}, \ \ell \geq 2,\\
\bar{\mathcal O}_{n^k,\ve} & \mbox{if} \ n, k\ \mbox{even}, \ell=0.
\end{cases}
\quad \mbox{if}\ \ve=1, v\ \mbox{is even}. 
\end{split}
\end{align}
Note that all orbit closures in (\ref{N-irr}) are irreducible except the last one parametrized by  a very even partition. 
In the latter case, the orbit closure is a union of two irreducible components of pure dimension. 
This concludes the analysis of the irreducibility of $\N_{n, v,\ve}$. 
Further, the description (\ref{N-irr}) can be used to determine the normality of $\N_{n, v,\ve}$.
As such, we have 

\begin{prop}
\label{N-irr-norm}
The $n$-nilcone $\N_{n, v,\ve}$ is irreducible and normal except the very even case:  $\ve=1$, $v=kn$ and $k, n $ are even. 
In the very even case, $\N_{n, v,\ve}$ is a union of two normal irreducible components of pure dimension.
\end{prop}

\begin{proof}
The irreducibility property follows from the above analysis. 

For normality, one can use Kraft-Procesi's $\ve$-degeneration criterion~\cite[Theorem 16.2]{KP82} to verify. 
For $\ve=-1$, the cases are verified in~\cite[Table 5, I-III]{AM18}. 
For $\ve=1$ and $v$ odd, the cases are verified in~\cite[Table 6, I-IV]{AM18}.
For $\ve=1$ and $v$ even, the first two  are verified in~\cite[Table 7, I, II]{AM18}. 
Let us verify the remaining cases. For the case $n, \ell $ even and $k$ odd,
the local singularity is of type $c$ if $n\leq \ell+4$ or of type $b$ if $n\geq \ell+4$. Hence the orbit closure is normal.
For the case $n, k, \ell$ even with $\ell \geq 2$, 
the local singularity is of type $d$ if $\ell =2$, of type $a$ if $\ell =4$, and of type $b$ if $\ell >4$. 
This again shows that the orbit closure is normal. 
It is known that each irreducible component of nilpotent orbit closure of a very even partition of form $n^k$ is normal. 
The proof is finished. 
\end{proof}

\begin{rem}
(1). Note 
that singularities in type IV, V  in $\mathfrak{sp}_{2d}$ case in~\cite{AM18}, not showing up in our setting,  is the same as that of I and II in {\it loc. cit.}, respectively, in light of~\cite[8.5]{Li19a}. 
So, essentially, all cases studied in~\cite[5.1]{AM18} for $\mathfrak{sp}_{2d}$ and $\mathfrak{so}_{2d+1}$ appear in our setting, mysteriously. 
It will be very interesting to relate
the theory of (quantum) symmetric pairs with the theory of W-algebras.

(2). When $n$ is a prime, the $n$-nilcone $\N_{n,v,\ve}$  appeared in the work~\cite{BNPP}.
\end{rem}

Let $x$ be a nilpotent element in $\mrm{Lie} (\G_{v,\ve})$. Let us fix an $\mathfrak{sl}_2$-triple $(x,y, h)$.
Let $S_x=x+\ker \mrm{ad} y$ be the associated Slodowy slice. We set
\begin{align}
\label{S}
S= \N_{n, v,\ve} \cap S_x,
\end{align}
to be the Slodowy slice of the $n$-nilcone at $x$. 
Since $\N_{n,v,\ve}$ is always a nilpotent orbit closure,
the variety $S$ is  a so-called nilpotent Slodowy slice.  
We are interested in the irreducibility of $S$. 
As  shown in~\cite[Lemma 3.1]{AM18}, $S$ is irreducible if $\N_{n,v,\ve}$, as an orbit closure,  is normal. 
Hence we have the following characterization, thanks to Proposition~\ref{N-irr-norm}. 

\begin{prop}
\label{S-irr}
The slice $S$ is irreducible except the case $\ve=1$,   $v=kn$ with $n,k$  even. 
\end{prop}

When $\ve=1$ and  $v=nk$ with $n,k$  even,
$\N_{n, v,\ve}$ is a union of two normal irreducible components, say $\N^{(1)}_{n, v,\ve}$ and $\N^{(2)}_{n, v,\ve}$. 
One can consider 
\[
S^{(i)} = \N^{(i)}_{n, v,\ve} \cap S_x, \quad i=1, 2.
\]
By exactly the same argument as above, we have

\begin{prop}
\label{S-irr-2}
When $\ve=1$,   $v=kn$ with $n,k$  even, $S^{(i)}$, for $i=1,2$, is irreducible.
\end{prop}

We end this section with a remark that nilpotent Slodowy slices are not irreducible in general, see~\cite[2.4]{FJLS15}.

\subsection{Stabilization of  $\N_{n, v, \ve}$}
\label{Stabilization}

With Propositions~\ref{S-irr} and~\ref{S-irr-2} in hand, we are ready to discuss the stabilization of $\N_{n, v,\ve}$ as $v\to \infty$.
Our treatment follows closely the approaches in~\cite[4.4]{CG97} for  the $\G_v$ case with a new insight from ~\cite{LW18}. 

Let $e$ be an  $n\times n$ Jordan block and $J$ be an $n\times n$ anti-diagonal matrix as follows. 
\begin{align}
\label{e}
e=
\begin{bmatrix}
0 & 1 & 0 & \cdots  & 0\\
0 & 0 & 1 & \cdots & 0 \\
&&& \cdots\\
0 & 0 & 0 & \cdots & 1\\
0 & 0 & 0 & \cdots & 0
\end{bmatrix}
\quad
J = 
\begin{bmatrix}
0 & \cdots & 0 & 0 &1\\
0 & \cdots & 0 & 1 & 0\\
&\cdots\\
0 &  \cdots  & 0 & 0 & 0\\
1 & \cdots  & 0 & 0& 0
\end{bmatrix}.
\end{align}
On $\mbb C^{2n}$, we define an $\ve$-form by the following matrix
\begin{align}
\label{Me}
M_{\ve} =
\begin{bmatrix}
0 & J \\
\ve J & 0 
\end{bmatrix}.
\end{align}
Implicitly, we fix a basis $\{ a_i, b_i\}_{1\leq i\leq n}$  of $\mbb C^{2n}$ such that $(a_i, b_j)=\delta_{ij}$,
$(a_i, a_j) =0=(b_i,b_j)$ for all $1\leq i, j \leq n$. 
Let
\[
e_{\ve} = 
\begin{bmatrix}
e & 0 \\
0 & - e
\end{bmatrix}.
\]
Then $e_\ve \in \mrm{Lie} (\G_{2n,\ve})$ and $e_{\ve}$ is nilpotent of type $n^2$.

We define a form on $V\oplus \mbb C^{2n}$ to be the direct sum of the forms on $V$ and $\mbb C^{2n}$. Then the assignment
$x\mapsto x\oplus e_{\ve}$ defines a closed immersion
\begin{align}
\label{i-N}
i_{\ve} : \N_{n, v, e}\to \N_{n, v+2n, \ve}, \quad x\mapsto x\oplus e_{\ve}. 
\end{align}
Similarly, the rule $g\mapsto g\oplus 1$ defines a group homomorphism $\G_{v,\ve} \to \G_{v+2n, \ve}$. 
By indexing appropriately the irreducible components, we may assume in the very even case that 
\begin{align}
\label{i-even}
i_{\ve} (\N^{(i)}_{n, v, \ve}) \subseteq \N^{(i)}_{n, v+2n, \ve}, \quad \forall i =1, 2. 
\end{align}

Following Chriss-Ginzburg~\cite{CG97}, if $\mathcal O=\mathcal O_x$ is a nilpotent orbit in $\N_{n, v,\ve}$, 
we write  $\mathcal O^{\dagger}= \mathcal O_{i_{\ve} (x)}$.  
If $\mathcal O$ is very even, we define similarly $(\mathcal O^{(i)})^{\dagger}$.
By the same argument as the proof of Lemma 4.4.4 in {\it loc. cit.} and  (\ref{O-order}) and (\ref{i-even}), we get
the following analogue. 

\begin{lem}
We have 
$$i_{\ve} (\mathcal O) = \mathcal O^{\dagger} \cap i_{\ve} ( \N_{n, v,\ve}).$$ 
Moreover if $\mathcal O$ is very even, then
$$i_{\ve} (\mathcal O^{(i)}) = (\mathcal O^{(i)})^{\dagger} \cap i_{\ve} ( \N_{n, v,\ve}), \forall i=1, 2.$$
If $\mathcal O' $ is a nilpotent orbit in $\N_{n, v+2n,\ve}$, then 
$\mathcal O^{\dagger} < \mathcal O'$ 
if and only if
there exists an orbit $\mathcal O_1\subseteq \N_{n, v, \ve}$ such that $\mathcal O'= \mathcal O_1^{\dagger}$ and $\mathcal O <\mathcal O_1$. 
\end{lem}

In light of the dimension formula (\ref{dim-O}), we have

\begin{lem}
The dimension difference of $\mathcal O^{\dagger}$ and $\mathcal O$ is independent of the partition type.
\[
\dim \mathcal O^{\dagger} - \dim \mathcal O = \dim \mathcal O_{i_\ve (0)} = 2 ( n+v) (n-1) - \ve (n- \delta_{1, \bar n}), 
\]
where $\bar n$ is the parity of $n$. 
\end{lem}

\begin{proof}
If $\mathcal O$'s partition type is $\mu$, then the partition type of $\mathcal O^{\dagger}$ is $(n^2, \mu)$.
Applying (\ref{dim-O}), we see that the dimension of $\mathcal O^{\dagger}$ is equal to
\begin{align*}
\dim \mathcal O^{\dagger} & = \frac{1}{2}
\left (
(2n+v)^2 - \ve (2n+v) - \sum_i (\hat \mu_i+2)^2 + \ve \# \{ i | \mu_i \mbox{is odd} \}  + \ve 2 \delta_{1,\bar n}
\right )\\
& =
\frac{1}{2}
\left (
v^2 -\ve v + 4nv + 4n^2   - \ve 2n  - \sum_i \hat \mu^2_i - 4 v - 4n  + \ve \# \{ i | \mu_i \mbox{is odd} \}  + \ve 2 \delta_{1,\bar n}
\right )\\
& = \dim \mathcal O +2 ( n+v) (n-1) - \ve (n- \delta_{1, \bar n})\\
&= \dim \mathcal O + \dim \mathcal O_{i_{\ve}(0)},
\end{align*}
where the last equality is from (\ref{dim-O}) and the fact that the partition type of $i_{\ve}(0)$ is  ($n^2, 1^{v}$).
\end{proof}

Let us fix an $\mathfrak{sl}_2$-triple $(i_\ve(x), y, h)$ as the direct sum of 
$\mathfrak{sl}_2$-triples of $x$ and $e_{\ve}$. 
Let $S_{i_{\ve}(x)}$ be the associated Slodowy slice. Similar to (\ref{S}), we consider the nilpotent Slodowy slice
\begin{align}
S^{\dagger} = \N_{n,v+2n,\ve} \cap S_{i_{\ve}(x)}. 
\end{align}

The following proposition shows that $S$ in (\ref{S}) is isomorphic to $S^{\dagger}$ with compatible stratifications by nilpotent orbits.

\begin{prop}
\label{S=S}
We have $i_{\ve} (S) = S^{\dagger}$. 
For any orbit $\mathcal O_1\subseteq \N_{n, v,\ve}$, there is an equality
$i_\ve (S\cap \mathcal O_1) = S^{\dagger} \cap \mathcal O_1^{\dagger}$.
\end{prop}

\begin{proof}
If $\N_{n, v,\ve}$ is not very even, the proof of~\cite[4.4.9]{CG97} applies here verbatim in light of  Proposition~\ref{S-irr}. 
In the very even case, the same argument shows that  $i_{\ve} (S^{(i)})= (S^{(i)})^{\dagger}$ for $i=1, 2$, thanks to Proposition~\ref{S-irr-2}. 
Therefore, there is
\[
i_{\ve}(S) = i_{\ve}(S^{(1)}\cup S^{(2)}) = (S^{(1)})^{\dagger} \cup (S^{(2)})^{\dagger}=(S)^{\dagger}.
\]
In a similar manner, one can show the remaining claim in this case. The proof is finished.
\end{proof}

The following is the main result in this section, which is an analogue of~\cite[4.4.16]{CG97}.

\begin{prop}
\label{U}
There is an open neighborhood $U\subseteq \N_{n,v+2n,\ve}$ of  $i_{\ve}(\N_{n,v,\ve})$, with respect to the analytic topology, such that
\[
U \cong (\mathcal O_{i_{\ve}(0)} \cap U ) \times i_{\ve} (\N_{n,v, \ve}). 
\]
Moreover, this isomorphism is compatible with the stratifications defined by  
nilpotent orbits: for any nilpotent orbit $\mathcal O \subseteq \N_{n, v,\ve}$, the above isomorphism restricts  to an isomorphism
\[
U\cap \mathcal O^{\dagger} \cong (\mathcal O_{i_{\ve}(0)} \cap U ) \times i_{\ve} (\mathcal O). 
\]
\end{prop}

\begin{proof}
The nilpotent orbits provides an algebraic stratifications of $\N_{n, v, \ve}$ and $\N_{n, v+2n,\ve}$.
The variety $\N_{n, v,\ve}$ is still a cone and moreover $i_{\ve}(S) =S^{\dagger}$ by Proposition~\ref{S=S}. With these facts in hand, the proof of ~\cite[4.4.16]{CG97} applies here verbatim.
\end{proof}

Note that we have a direct system $(\N_{n, v+ 2nk, \ve}, i_{\ve})_{k\in \mbb N}$. Let
\begin{align}
\N_{n, v+\infty, \ve} = \varinjlim_k \ (\N_{n, v+ 2nk,\ve}, i_{\ve}).
\end{align}
be the ind-variety. We refer to~\cite[IV]{K02} for an introduction to the theory of ind-varieties.
Similarly, we have the direct limit
\[
\G_{v+\infty, \ve} =\varinjlim_k \G_{v+2nk, \ve},
\]
which is an ind-group.
There is a natural $\G_{v+\infty,\ve}$ action on $\N_{n, v+\infty, \ve}$. 
Define a map $\P_\ve(v) \to \P_\ve (v+2n) , \mu\mapsto (n^2,\mu)$. We have the direct limit
\[
\P_{\ve}(v+\infty)= \varinjlim_k \P_\ve(v+ 2nk).
\]
Let $\P_{\ve}(n, v+\infty)$ be the subset in $\P_{\ve}(v+\infty)$ consisting of all partitions of parts less than or equal to $n$.
Then the $\G_{v+\infty,\ve}$-orbits in $\N_{n, v+\infty,\ve}$ are parametrized by $\P_{\ve}(n,v+\infty)$.

\subsection{Stabilization of isotropic $n$-flag varieties and their cotangent bundles}
\label{M-stab}
Let $\F_{n, v,\ve}$ be the variety of  isotropic  $n$-step flags (or $n$-flag for short)  in $V$ of the form
\[
F = ( 0 \equiv F_0\subseteq F_1\subseteq \cdots \subseteq F_n\equiv V), 
\ F_i^{\perp} = F_{n-i} \forall 1 \leq i\leq n. 
\]
Note $\F_{n, v,\ve}$ is empty  for the case $(n, v,\ve) = (even, odd, 1)$; see (\ref{rectify}) for a new treatment.
There is a natural $\G_{v,\ve}$-action on $\F_{n,v,\ve}$.
The $\G_{v,\ve}$-orbits of $\F_{n, v,\ve}$ form  a partition  as follows.
\[
\F_{n,v,\ve} = \sqcup_{\mbf d\in \Lambda_{v}} \F_{\mbf d,\ve},
\quad
\F_{\mbf d,\ve} =\{ F\in \F_{n, v,\ve} | \dim F_i/F_{i-1} = d_i+ \delta_{i, r+1} \delta_{1,\ve} \delta_{-1,(-1)^v} \ \forall 1\leq i\leq n\}. 
\]

Recall from Section~\ref{Stabilization}, we fix a basis $\{a_i, b_i\}_{1\leq i\leq n}$ of $\mbb C^{2n}$ and define a form whose matrix under the basis is $M_{\ve}$ in (\ref{Me}). 
Fix the following $n$-step isotropic flag in $\F_{n, 2n,\ve}$:
\[
\mbf F_{\ve} =(
0 \subset \langle a_1,b_n\rangle \subset \langle a_1, a_2, b_n, b_{n-1}\rangle \subset \cdots \subset \langle a_1, \cdots, a_{n-1}, b_n,\cdots, b_2\rangle \subset \mbb C^{2n} 
). 
\]
For any flag $F\in \F_{n, v,\ve}$, let $F\oplus \mbf F_{\ve}$ be a flag in $\F_{n, v+2n,\ve}$ (with $V\oplus \mbb C^{2n}$ as the underlying space) whose $i$th step is the sum of the $i$th steps of $F$ and $\mbf F_{\ve}$. 
We define an embedding
\begin{align}
\label{i-F}
i_{\F}: \F_{n, v, \ve} \to \F_{n, v+2n, \ve}, 
\ F\mapsto F\oplus \mbf F_{\ve}.
\end{align}

Let 
\begin{align}
\label{Theta}
\Theta_{v+\infty} =
\left \{ 
A=(a_{ij})_{1\leq i,j\leq n} |  
\substack{
\sum_{1\leq i, j\leq n} a_{ij} \equiv v - \delta_{-1, (-1)^v} \ (\mbox{\tiny{mod}}\ 2n),\\  a_{ij} =a_{n+1-i,n+1-j} \in \mbb N
}
\right \}.
\end{align}
The set $\Theta_{v+\infty}$ admits a partition
\[
\Theta_{v+\infty} = \sqcup_{k: v+2nk>0} \Theta_{v+2nk}, \quad \Theta_{v+2nk} =\{ A\in \Theta_{v+\infty}| \sum_{i, j} a_{ij} = v-  \delta_{-1, (-1)^v} + 2kn\}.
\]
We define an equivalence relation on $\Theta_{v+\infty}$ by $A\sim B$ if and only if
$A\equiv B \ (\mbox{mod} \ 2I)$.
Let
\[
\overline \Theta_{v+\infty}= \Theta_{v+\infty}/\sim,
\]
and $\overline A$ be the equivalence class of $A$. 
We define two maps
\[
\ro, \co : \Theta_{v+\infty} \to \Lambda_{v+\infty},
\]
where the $i$-th entry of $\ro(A)$ (resp. $\co (A)$) is equal to 
$\sum_{l} a_{il}$ (resp. $\sum_{k} a_{ki}$).
Clearly these maps induce maps $\ro, \co : \bar \Theta_{v+\infty}\to \bar \Lambda_{v+\infty}$.

Note that the closed imbedding in (\ref{i-F}) defines a direct system $(\F_{n,v+2kn,\ve}, i_{\ve})$. 
Let
\[
\F_{n, v+\infty,\ve} =\varinjlim_k  \F_{n,v+2kn,\ve}
\]
be the ind-variety of the direct system.
The ind-group $\G_{v+\infty, \ve}$ acts transitively on $\F_{n,v+\infty,\ve}$, and induces a diagonal action on
the product $\F_{n,v+\infty,\ve} \times \F_{n,v+\infty,\ve}$. 

Given $[F, F'] \in \F_{n,v+\infty,\ve} \times \F_{n,v+\infty,\ve}$, we can define a matrix
$A_{F, F'}$ whose $(i, j)$-th entry is
\[ 
\dim \frac{F_i\cap F'_j}{ F_{i-1}\cap F'_j+F_i\cap F'_{j-1}}-
\delta_{i, r+1}\delta_{j, r+1} \delta_{-1,(-1)^v}. 
\]
It is easy to check that 
$A_{F, F'} +2I = A_{F\oplus \mbf F_\ve,F'\oplus \mbf F_{\ve}}$.
This shows that there is a well-defined map from the set of $\G_{v+\infty,\ve}$-orbits in
$\F_{n,v+\infty,\ve} \times \F_{n,v+\infty,\ve}$ to the set $\overline \Theta_{v+\infty}$. 

\begin{lem}
The rule 
$
[F, F']\mapsto  \overline A_{F, F'} 
$
defines 
a bijection 
$$\G_{v+\infty,\ve} \backslash \F_{n,v+\infty,\ve} \times \F_{n,v+\infty,\ve} \cong \overline \Theta_{v+\infty}.$$
\end{lem}

\begin{proof}
Let $\overline A\in \overline \Theta_{v+\infty}$. Let us fix a vector space $V$ of dimension $\sum_{i,j} a_{ij}$ with an $\ve$-form. 
By the symmetry on $A$, we can decompose $V$ as $V=\oplus_{ij} V_{ij}$ such that 
\begin{itemize}
\item[(a)] $\dim V_{ij} = a_{ij}+\delta_{i, r+1}\delta_{j, r+1} \delta_{-1,(-1)^v}$,
\item [(b)]
the  restriction of the $\ve$-form on $V_{ij} \oplus V_{n+1-i,n+1-j}$, if $(i,j) \neq (n+1-i, n+1-j)$,
$V_{ij}$ if $(i,j) = (n+1-i, n+1-j)$ is non-degenerate, and moreover
\item [(c)]
$V_{ij}$ is isotropic and dual to $V_{n+1-i,n+1-j}$  if $(i,j) \neq (n+1-i, n+1-j)$.
\end{itemize}
Let $F$ (resp. $F'$) be the isotropic flag whose $i$-th step is $\oplus_{k\leq i, 1\leq j\leq n} V_{kj}$
(resp. $\oplus_{1\leq k \leq n, j\leq i} V_{kj} $). 
Then we have $A_{F, F'} =A$. This shows that the map is surjective. 

Assume $(F, F')$ and $(\tilde F, \tilde F')$ are two pairs of isotropic flags such that the associated matrices are $A$.
To the pair $(F, F')$ (resp. ($\tilde F, \tilde F')$),  decompose $V$ as $V=\oplus_{i,j}V_{ij}$ (resp. $\oplus_{ij} \tilde V_{ij}$) subject to the conditions (a)-(c).
Define isomorphisms $g_{ij}: V_{ij}\to \tilde V_{ij}$ so that 
$(g_{ij}, g_{n+1-i, n+1-j})$ (resp. $g_{ij}$)  is compatible with the $\ve$-forms on $V_{ij}\oplus V_{n+1-i,n+1-j}$  and $\tilde V_{ij}\oplus \tilde V_{n+1-i,n+1-j}$
(resp. $V_{ij}$ and $\tilde V_{ij}$)  if $(i,j) \neq (n+1-i,n+1-j)$ (resp. o.w.).
Then $g=(g_{ij})\in \G_{\dim V, \ve}$ and $g(F, F') = (\tilde F, \tilde F')$. Thus the map is injective, and the proof is finished.
\end{proof}

We shall denote $Y_{\overline A, \ve}$ the $\G_{v+\infty,\ve}$-orbit in $\F_{n,v+\infty,\ve} \times \F_{n,v+\infty,\ve}$ indexed by
$\overline A$. 
If $A\in \Theta_{v}$, we set $Y_{A,\ve} = Y_{\overline A,\ve}\cap (\F_{n, v,\ve}\times \F_{n, v,\ve})$, a $\G_{v,\ve}$-orbit indexed by $A$.

Let $\M_{n, v,\ve}$ be the cotangent bundle of $\F_{n, v,\ve}$, which consists of all pairs $(x, F)\in \N_{n, v,\ve}\times \F_{n, v, \ve}$ such that
$x(F_i) \subseteq F_{i-1}$ for all  $i$.  Consider the first projection
\begin{align}
\pi: \M_{n, v, \ve} \to \N_{n, v, \ve}, \quad (x, F) \mapsto x. 
\end{align}

Observe that $(e_{\ve}, \mbf F_{\ve}) \in \M_{n, 2n, \ve}$.  The two embeddings  $i_{\ve}$ and $i_{\F}$ in (\ref{i-N}) and (\ref{i-F}) induce 
\begin{align}
\label{i-M}
i_{\M} : \M_{n, v,\ve} \to \M_{n, v+2n, \ve}, \ (x, F) \mapsto (x\oplus e_{\ve} , F\oplus \mbf F_{\ve}).
\end{align}
Clearly, we have a commutative diagram.
\begin{align}
\label{cartesian}
\begin{CD}
\M_{n, v,\ve} @>i_{\M} >> \M_{n, v+2n, \ve}\\
@V\pi VV @V\pi VV\\
\N_{n, v,\ve} @>i_{\ve} >> \N_{n, v+2n,\ve}
\end{CD}
\end{align}

\begin{lem}
\label{pi-cartesian}
The above diagram (\ref{cartesian})  is cartesian, i.e., 
$i_{\M} ( \pi^{-1}(x)) = \pi^{-1}( i_{\ve} (x))$. 
\end{lem}

\begin{proof}
Clearly, we have $ i_{\M} ( \pi^{-1}(x)) \subseteq \pi^{-1}( i_{\ve} (x))$.
For any nilpotent element $x$ such that $x^n=0$. 
Let $F^{min}(x)$ (resp. $F^{max}(x))$  be the $n$-step flag whose $i$-th step is equal to $\mrm{im } (x^i)$
(resp. $\ker (x^{n-i})$). 
We observe that $F^{max}(e_{\ve}) = \mbf F_{\ve} = F^{min}(e_{\ve})$.
Now, the argument in the proof of 
~\cite[4.4.23]{CG97} implies that if $(x\oplus e_{\ve}, F^{\dagger})$ is in $\M_{n, v+2n,\ve}$ then there exists $F\in \pi^{-1}(x)$ such that
$F^{\dagger} = F\oplus \mbf F_{\ve}$.  Therefore,  $ i_{\M} ( \pi^{-1}(x)) \supseteq \pi^{-1}( i_{\ve} (x))$, so the lemma follows.
\end{proof}

\begin{rem}
Lemma~\ref{pi-cartesian} implies that
the Spaltenstein varieties are isomorphic under Kraft-Procesi's row reduction. 
We conjecture that the statement remains true under Kraft-Procesi's column reduction. 
See also~\cite{Li19b}.
\end{rem}

We now show that a statement similar to Proposition~\ref{U} still holds for $\M_{n,v,\ve}$.
Let $U$ be the open neighborhood in Proposition~\ref{U}. We set
$\tilde U = \pi^{-1}(U)$. Then the  proof of~\cite[Proposition 4.4.26]{CG97} 
leads to  the following statement similar to Proposition~\ref{U}.

\begin{prop}
\label{phi-U}
There is an isomorphism 
\begin{align}
\label{Qa}
 \tilde U\overset{\cong}{ \to} (\mathcal O_{i_{\ve}(0)} \cap U) \times i_{\ve}(\M_{n, v,\ve}) 
\end{align}
fitting into the following commutative diagram 
\[
\xymatrix{
\M_{n, v,\ve} \ar@{^{(}->}[r]  \ar[d]^{\pi}& \tilde U \ar[d]^{\pi} \ar@{->}[r]  &   (\mathcal O_{i_{\ve}(0)} \cap U) \times i_{\ve}(\M_{n, v,\ve}) \ar[d]^{\mrm{id}\times \pi} \\
\N_{n, v, \ve} \ar@{^{(}->}[r]   & U \ar[r]& (\mathcal O_{i_{\ve}(0)} \cap U) \times i_{\ve}(\N_{n, v,\ve}) 
}
\]
\end{prop}

Finally, we record the dimension difference of $T^* \F_{\mbf d,\ve}$ and $T^*\F_{\mbf d+(2,\cdots, 2),\ve}$ for later usage.
It is known that 
$
\dim \F_{\mbf d,\ve} = \frac{1}{2} ( \sum_{i<k} d_i d_k - (-1)^v \ve \sum_{i< r+1} d_i).
$
So 
\begin{align}
\label{dim-F}
\dim T^* \F_{\mbf d+(2,\cdots, 2),\ve } - \dim T^* \F_{\mbf d,\ve}  =\dim \mathcal O_{i_{\ve}(0)}= 2 ( n+v) (n-1) - \ve (n- \delta_{1, \bar n}).
\end{align}
The advantage of the formula (\ref{dim-F}) is that we can write
\[
\dim \M_{n, v+2n,\ve} - \dim \M_{n, v, \ve} = \dim \mathcal O_{i_{\ve}(0)},
\]
even though we need to fix a connected component for the varieties involved.

\subsection{Steinberg ind-varieties of classical type}
\label{Z-stab}

Let
\[
Z_{n, v,\ve} = \M_{n, v, \ve} \times_{\N_{n, v,\ve}} \M_{n, v, \ve} =\{ (x, F, F') | (x, F), (x, F') \in \M_{n, v, \ve}\}
\]
be the $n$-step Steinberg variety of classical type.
For convenience, we shall simply call $Z_{n,v,\ve}$ a Steinberg variety.  
It is a Lagrangian subvariety of the symplectic manifold $\M_{n,v,\ve}\times \M_{n,v,\ve}$, up to  a conventional twist on the symplectic structure. 
The embedding $i_{\M}$ in (\ref{i-M}) induces the following cartesian diagram 
\begin{align}
\label{i-Z-1}
\begin{split}
\xymatrix{
 Z_{n, v, \ve} \ar@{->}[r]^{i_z}  \ar@{^{(}->}[d] &  Z_{n, v+2n, \ve} \ar@{^{(}->}[d]\\
 \M_{n,v,\ve}\times \M_{n, v+2n,\ve}  \ar@{->}[r] & \M_{n, v+2n,\ve}\times \M_{n, v+2n,\ve}
}
\end{split}
\end{align}
where the left vertical map is the composition
\[
Z_{n, v,\ve} \hookrightarrow \M_{n, v, \ve}\times\M_{n,v,\ve} \overset{1\times i_{\M}}{\hookrightarrow} \M_{n,v,\ve}\times \M_{n, v+2n,\ve}.
\]

Let $Z_{A,\ve}=\{ (x, F, F')\in Z_{n, v,\ve}| (F, F)\in Y_{A,\ve}\}$ for any $A\in \Theta_v$. 
Then  its closure  $\overline Z_{A,\ve}$ is an irreducible component of $Z_{n, v,\ve}$ if $(v,\ve)\neq (\mbox{even}, 1)$.
If $(v,\ve)=(\mbox{even}, 1)$, then $\overline Z_{A,\ve}$ may split into a union of irreducible components in $Z_{n, v,\ve}$. 
Clearly, one has
$
i^{-1}_z (Z_{A,\ve}) = 
Z_{A-2I,\ve}$  if $A-2I\in \Theta_{v}$
or $\O$ otherwise. 
Hence, we have
$
\overline Z_{A-2I,\ve} \subseteq i^{-1}_{z} (\overline Z_{A,\ve}).
$
Moreover the result can be refined  as follows.  

\begin{lem}
\label{i-Z}
For any $A\in \Theta_{v+2n}$, any irreducible component of 
$ i^{-1}_z (\overline Z_{A,\ve})$ of dimension $\dim Z_{n,v,\ve}$ is either
those in $\overline Z_{A-2I,\ve}$ if  $A-2I\in \Theta_{v}$ or nonexistent otherwise.
\end{lem}

\begin{proof}
It is known that $Z_{n,v,\ve}$ is pure dimensional and its irreducible components appear as the irreducible components in the closure of
$Z_{\mathcal O} =\{(x, F, F')\in Z_{n, v,\ve}| x\in \mathcal O\}$ for a nilpotent orbit in $\N_{n,v,\ve}$. 
Now $Z_{\mathcal O}$ is isomorphic to $\G_{v,\ve}\times_{ \G_{v,\ve}(x)} ( \pi^{-1}(x)\times\pi^{-1}(x))$ for a fixed $x\in \mathcal O$ and $\G_{v,\ve}(x)$ is the stabilizer of $x$ in $\G_{v,\ve}$.
A typical irreducible component would be the closure of
$\G^o_{v,\ve}\times_{\G^o_{v,\ve}(x)} \Lambda$ where the superscript $o$ denote the identity component and $\Lambda$ is an irreducible component in $\pi^{-1}(x)\times\pi^{-1}(x)$.
By Lemma~\ref{pi-cartesian}, the fiber $\pi^{-1}(x)$ remains unchanged under $i_{\M}$. 
So if there were an irreducible component of $i^{-1}_{z} (\overline Z_{A,\ve})$ of dimension $\dim Z_{n,v,\ve}$  other than those in $\overline Z_{A-2I,\ve}$,
there would be  an irreducible component of dimension $\dim Z_{n, v+2n,\ve}$ in $\overline Z_{A,\ve}$ other than those in $\overline Z_{A,\ve}$, which is absurd. 
The lemma is thus proved.
\end{proof}

\begin{rem}
One expects a stronger result than Lemma~\ref{i-Z}, i.e.,
\begin{align}
\label{i-Z-2}
 i^{-1}_{z} (\overline Z_{A,\ve})
 =
 \begin{cases}
 \overline Z_{A-2I,\ve} & \mbox{if} \ A-2I\in \Theta_{v},\\
 \O & \mbox{o.w}.
 \end{cases}
\end{align}
This is similar to~\cite[Cor. 4.4.29]{CG97}, but whose proof therein does not seem to apply here.
In any case, Lemma~\ref{i-Z} is enough for us to proceed. 
\end{rem}

Define the direct limit of  $(Z_{n, v+2nk,\ve}, i_{z})_{k\in \mbb N}$ by
\[
Z_{n, v+\infty,\ve} = \varinjlim_k (Z_{n, v+2nk,\ve},i_{z}).
\]
Note that $Z_{n,v+\infty,\ve}$ can be identified with the equivalence class
$\sqcup_k Z_{n, v+2nk,\ve}/\sim $, where the equivalence $\sim$ is defined by
for all $(x, F, F')\in Z_{n, v+2nk,\ve}$ and $(y, G, G')\in Z_{n, v+2k'n,\ve}$ for $k<k'$,
$(x, F, F')\sim (y, G, G')$ if and only if $i_{k', k}(x, F, F') =(y, G, G')$ where
$i_{k', k} : Z_{n, v+2kn, \ve} \to Z_{n, v+2k'n,\ve}$ for $k<k'$ is the transfer map, induced by $i_{z}$,  in the direct system $(Z_{n, v+2nk,\ve}, i_{z})_{k\in \mbb N}$. 
Let $[x, F, F']$ be the equivalence class of $(x, F, F')$. 
Let $i_k: Z_{n, v+2kn, \ve} \to Z_{n, v+\infty,\ve}$ be the induced embedding from the direct system $(Z_{n, v+2nk,\ve}, i_{\ve})_{k\in \mbb N}$. 

Consider the following   $n$-$step$ $Steinberg$ $ind$-$varieties$ $o \! f$ $classical$ $type$.
\begin{align}
\label{Z-infty-a}
Z_{n, \infty,\ve} : = \sqcup_{v\in I_{n, \ve}} Z_{n, v, \ve}
%
%
\end{align}
where $I_{n, \ve}$ is defined in (\ref{I}).
The direct limit $Z_{n,\infty,\ve}$ admits a Zariski topology induced from those in $Z_{n,v+2kn,\ve}$ for all $1\leq v\leq 2n$, $k\in \mbb N$.
Precisely, we claim that a subset $X\subseteq Z_{n,\infty, \ve}$ is open (resp. constructible) if the restriction $i^{-1}_k(X)$ is Zariski open (resp. constructible) in $Z_{n,v+2kn,\ve}$ for all $k$ and $v$. 
Then it is clear that the collection of all open sets in $Z_{n, \infty, \ve}$ forms the desired Zariski topology.

We are interested in describing  the irreducible components of $Z_{n,\infty,\ve}$.
Let $p: Z_{n,v+\infty,\ve} \to \F_{n,v+\infty,\ve} \times \F_{n,v+\infty,\ve}$ be the projection map
$[x, F, F']\mapsto [F, F']$. 
Let 
\[
Z_{\overline A,\ve} = p^{-1} ( Y_{\overline A,\ve}).
\]
Clearly we have 
\[
Z_{n,v+\infty,\ve} = \sqcup_{\overline A \in \overline \Theta_{v+\infty}} Z_{\overline A,\ve}. 
\]
Let $\overline Z_{\overline A,\ve}$ be the closure of $Z_{\overline A,\ve}$ in $Z_{n, v+\infty,\ve}$.
One can show that 

\begin{rem}
\label{Z-infty}
If (\ref{i-Z-2}) holds, then the subsets $\overline Z_{\overline A,\ve}$ for $\overline A\in \overline \Theta_{v+\infty}$ 
form the list of irreducible components of $Z_{n,\infty,\ve}$. 
The proof goes as follows.A closed subset $X$ in $Z_{n,\infty,\ve}$ is irreducible if and only if $i^{-1}_k(X)$ is irreducible for $k>>0$. Now apply the assumption (\ref{i-Z-2}) to obtain the claim. 
\end{rem}

\subsection{Borel-Moore homology and transfer maps}
\label{BM-TM}
Let $\H_*(X)$ be the Borel-Moore homology for a complex algebraic variety $X$ with rational coefficients. 
We write $\H_{irr}(X)$ for the subspace of $\H_*(X)$  spanned by the fundamental classes of irreducible components in $X$. 
An introduction to this homology theory can be found in~\cite[2.6]{CG97}. 
By applying restriction with supports~\cite[2.6.21]{CG97} on (\ref{i-Z-1}), we have a map
\begin{align}
\label{phi-1}
\phi_{v+2n, v} : \H_* (Z_{n, v+2n,\ve}) \to \H_{*-2\dim \mathcal O_{i_{\ve}(0)}} (Z_{n, v, \ve}).
\end{align}
We shall call $\phi_{v+2n,v}$ a transfer map. 

\begin{prop}
\label{stab}
The map
$
\phi_{v+2n, v} 
$
in (\ref{phi-1})
satisfies that 
\[
\phi_{v+2n, v} ( [ \overline Z_{A+ 2I,\ve}] ) = 
\begin{cases}
[\overline Z_{A,\ve}] & \mbox{if} \  A\in \Theta_{v}, \\
0 & \mbox{if}\ A\not\in \Theta_v.
\end{cases}
\] 
\end{prop}

\begin{proof}
This follows from Lemma~\ref{i-Z} and the definition of $\phi_{v+2n, v}$. 
\end{proof}

By applying the machinery built up in~\cite[2.7]{CG97}, the space $\H_*(Z_{n, v,\ve})$, equipped with the convolution product, is a unital  associative algebra.  
Consider the intersection cohomology complex $\IC_{\M_{n, v,\ve}}$ with coefficients in $\mbb Q$ and 
\begin{align}
\label{L}
L_{n, v,\ve} = \pi_* (\IC_{\M_{n,v,\ve}}).
\end{align}
Since $\pi$ is semismall, we see that $L_{n,v,\ve}$ is a semisimple perverse sheaf. 
As associative algebras,
$$\H_*(Z_{n,v,\ve})\cong \Ext^* (L_{n,v,\ve}, L_{n, v,\ve}).$$ 
By (\ref{cartesian}) and (\ref{dim-F}), we have immediately
a canonical isomorphism $
i^*_{\ve} [-a_{\ve}] (L_{n, v+2n,\ve}) \cong L_{n, v, \ve}$, where $a_{\ve}=\dim \mathcal O_{i_{\ve}(0)}$.
In particular, the functor $i^*_{\ve}[-a_{\ve}]$  defines an algebra homomorphism
\begin{align}
\label{i-homo-1}
i^{*}_{\ve}[-a_{\ve}]: \Ext^*(L_{n,v+2n,\ve}, L_{n, v+2n,\ve}) \to \Ext^* (L_{n,v,\ve}, L_{n, v,\ve}).
\end{align}
Further we have that the two homomorphisms (\ref{phi-1}) and (\ref{i-homo-1}) coincide. 
\begin{prop}
\label{phi-homo}
We have that  $\phi_{v+2n,v}=i^{*}_{\ve}[-a_{\ve}]$  and hence $\phi_{v+2n,v}$ is    an algebra homomorphism. 
\end{prop}

\begin{proof}

To proceed, we need to restrict to connected components of  $\M_{n,v,\ve}$, and for simplicity, and by an abuse of notation,  we identify $\M_{n, v,\ve}$ with one of its connected components. 
Expand the diagram (\ref{i-Z-1}):
\[
\xymatrix{
& \M_{n,v,\ve}\times \M_{n, v+2n,\ve} \ar@{->}[rr]^{j_z} \ar@{.>}[dd] &  & \M_{n, v+2n,\ve}\times \M_{n, v+2n,\ve} \ar@{->}[dd]\\
Z_{n, v,\ve}  \ar@{->}[rr] \ar@{->}[ur] \ar@{->}[dd] & & Z_{n, v+2n,\ve} \ar@{->}[ur]_{\Delta_z} \ar@{->}[dd]^(.25){\pi^2} \\
& \N_{n, v,\ve}\times \N_{n, v+2n, \ve} \ar@{.>}[rr] & & \N_{n, v+2n,\ve} \times \N_{n, v+2n, \ve} \\
\N_{n, v,\ve}\ar@{.>}[ur] \ar@{->}[rr]^{i_\ve} &  &\N_{n, v+2n, \ve}\ar@{->}[ur]
}
\]
where the maps are inclusion or projections. Observe that all diagrams are cartesian.
So by proper base change, there are
\begin{align}
\label{transfer-a}
\begin{split}
\pi^2_* \Delta_z^! \IC_{\M_{n, v+2n,\ve}\times \M_{n,v+2n,\ve} }
& \cong \mathcal H \mrm{om} (L_{n, v+2n,\ve}, L_{n, v+2n,\ve}),  \quad ([\mbox{CG97, 8.6.4}])\\
\pi^2_* \Delta_z^! j_{z*} j_z^* \IC_{\M_{n, v+2n,\ve}\times \M_{n,v+2n,\ve} }
& \cong i_{\ve*} i_{\ve}^! [2a] \pi^2_* \Delta^!_z  \IC_{\M_{n, v+2n,\ve}\times \M_{n,v+2n,\ve} }  \\
& \cong i_{\ve*} \mathcal{H}\mbox{om} (i^*_{\ve} [-a_{\ve}] (L_{n, v+2n,\ve}), i^*_{\ve} [-a_{\ve}] (L_{n, v+2n,\ve})) .
\end{split}
\end{align}
Note that 
\begin{align}
\label{transfer-b}
\begin{split}
\H_{j} (Z_{n, v+2n,\ve}) &\cong \H^{-j}( Z_{n, v+2n,\ve}, \mbb D_{Z_{n, v+2n,\ve}}) \\
&\cong \H^{-j}(\N_{n, v+2n, \ve}, \pi^2_* \Delta^!_z \mbb D_{\M_{n, v+2n,\ve}\times \M_{n,v+2n,\ve} })\\
&\cong \H^{-j}(\N_{n, v+2n, \ve}, \pi^2_* \Delta^!_z \IC_{\M_{n, v+2n,\ve}\times \M_{n,v+2n,\ve} } [2b])\\
&\cong  \H^{-j}(\N_{n, v+2n, \ve}, \mathcal H \mrm{om} (L_{n, v+2n,\ve}, L_{n, v+2n,\ve})[2b])\\
& \cong \mrm{Ext}^{-j+2b}(L_{n, v+2n,\ve}, L_{n, v+2n,\ve}),
\end{split}
\end{align}
where $b=\dim \M_{n, v+2n,\ve}$.
Let 
\[
\mrm{adj}: 1\to j_{z*} j_z^*
\]
be the adjunction of the adjoint pair $(j^*_z,j_{z*})$. Using (\ref{transfer-a}) and applying $\pi^2_* \Delta^!_z$, we get 
\[
\pi^2_*\Delta_z^! (\mrm{adj}):  \mathcal H \mrm{om} (L_{n, v+2n,\ve}, L_{n, v+2n,\ve})
\to i_{\ve*} \mathcal{H}\mbox{om} (i^*_{\ve} [-a_{\ve}] (L_{n, v+2n,\ve}), i^*_{\ve} [-a_{\ve}] (L_{n, v+2n,\ve})).
\]
For convenience, we write $\End^j(A)$ for $\mrm{Ext}^j(A, A)$. 
By~\cite[8.3.21]{CG97}, we know from (\ref{transfer-a})-(\ref{transfer-b}) that the restriction with supports $\phi_{v+2n,v}$  gets identified with  the map
\[
\H^{-*} ( \pi^2_*\Delta_z^! (\mrm{adj})): 
\mrm{End}^{-*+2b}(L_{n, v+2n,\ve})
\to 
\mrm{End}^{-*+2b}(i^*_{\ve} [-a_{\ve}] (L_{n, v+2n,\ve})),
\]
which is nothing but $i^*_{\ve}[-a_{\ve}]$. Proposition is thus proved. 
\end{proof}

\begin{rem}
(1). The  limit of the projective system $(\H_*(Z_{n,v,\ve}), \phi_{v+2n,v})$ can be thought of as the Borel-Moore homology of $Z_{n,\infty,\ve}$. 

(2). It is desirable to have a sheaf-free proof  that $\phi_{v+2n,v}$ is an algebra homomorphism.
\end{rem}

Let 
\begin{align}
\label{H-top}
\H(Z_{n,v,\ve})= \mrm{span}_{\mbb Q} \{ [\overline Z_{A,\ve}]| A\in \Theta_v\}.
\end{align}
Note that  if $(v,\ve)\neq (\mbox{even}, 1)$, then
$\H(Z_{n,v,\ve})$ is the top Borel-Moore homology of $Z_{n,v,\ve}$ spanned by the fundamental class of irreducible components of $Z_{n, v,\ve}$. 
If $(v,\ve)=(\mbox{even},1)$, it is a subspace of that of $Z_{n,v,\ve}$ because some fundamental class $[\overline Z_{A,\ve}]$ is a sum of fundamental classes of irreducible components therein.
Note that $\G_{v,\ve}$ acts naturally on $\H_{irr}(Z_{n, v, \ve})$. 
In particular, we have
\begin{align}
\label{Z-G-fix}
\H(Z_{n, v,\ve}) = \H_{irr} (Z_{n, v, \ve})^{\G_{v,\ve}}=\H_{irr}(Z_{n, v,\ve})^{\mbb Z/2\mbb Z}, \quad \mbox{if} \ (v,\ve) = (\mbox{even}, 1),
\end{align}
where we use $\G_{v,\ve}/\G_{v,\ve}^o\cong \mbb Z/2\mbb Z$.
By general machinery of top Borel-Moore homology, we know that $\H(Z_{n, v,\ve})$ is automatically a subalgebra of $\H_*(Z_{n,v,\ve})$ if $(v, \ve) \neq (\mbox{even}, 1)$.
If $(v, \ve) = (\mbox{even}, 1)$, then the $\G_{v,\ve}$-equivariance of $\overline Z_{A,\ve}$ 
implies that
$\H(Z_{n,v,\ve})$ is a subalgebra of $\H_*(Z_{n, v,\ve})$ as well. 
 
Proposition~\ref{stab} then implies that $\phi_{v+2n,v}$ descends to a surjective algebra homomorphism
\begin{align}
\label{phi-top}
\phi_{v+2n,v}: \H(Z_{n,v+2n,\ve}) \to \H(Z_{n, v,\ve}).
\end{align}

\subsection{Variants}
\label{variant}
To a large extent, the stabilization in the previous sections reflects the embedding
$\P_{\ve}(v)\to \P_{\ve}(v+2n),\ \mu\mapsto (n^2, \mu)$.  There is yet another   natural embedding
\[
\P_{\ve} \to \P_{\ve}(v+n), \mu \mapsto (n, \mu),\ \mbox{if} \ \ve=1, n \ \mbox{is odd, or}, \mbox{if}\ \ve =-1, n, \mbox{is even}.\\
\]
It leads to a stabilization similar to the previous case. In what follows, we briefly describe how to modify the previous approach to this setting.

If $\ve =1$ and $n$ is odd so that $n=2r+1$. We  fix a basis on $\mbb C^n$ so that the associated orthogonal form 
is defined by the matrix $J$. Let
\[
e_\ve =\begin{bmatrix}
e & Jf^T & 0\\
0 & 0 & - f\\
0 & 0 & -e
\end{bmatrix}
\]
where $e$ is defined in (\ref{e}) of size $r\times r$ and $f= [1,0,\cdots,0]$. 
Note that $e_{\ve}$ is almost a Jordan block, except that some of the $1$'s are now $-1$. 
Note that there is a unique $n$-step isotropic flag, say $\mbf F_{\ve}$, fixed by $e_{\ve}$. 
Similar results in Sections~\ref{Stabilization},~\ref{M-stab} and~\ref{Z-stab} remain valid in this setting. In particular, there is an algebra homomorphism
\begin{align}
\label{phi-n}
\phi_{v+n, v} : \H(Z_{n,v+n,\ve}) \to \H(Z_{n, v,\ve}), \quad\mbox{if}\ (n,\ve)=(\mbox{odd}, 1).  
\end{align}

Now we consider the case $\ve=-1$ and $n$ is even so that $n=2r$. 
In this case, we fix a basis on $\mbb C^n$ so that its symplectic form  is given by $M_{\ve}$  in (\ref{Me}).
Let 
\[
e_\ve=
\begin{bmatrix}
e & E_{r1} \\
0 & -e
\end{bmatrix}
\]
where $E_{r1}$ the matrix whose $(i, j)$ entry is $\delta_{i,r}\delta_{j, 1}$. 
In particular $e_{\ve}$ is almost a Jordan block except that some of the $1$'s are $-1$.
So there is a unique $n$-step isotropic flag fixed by $e_\ve$. 
Results similar to the previous sections still hold in this setting. 
In particular, there is an algebra homomorphism
\begin{align}
\label{phi-n-2}
\phi_{v+n, v} : \H(Z_{n,v+n,\ve}) \to \H(Z_{n, v,\ve}), \quad\mbox{if}\ (n,\ve)= (\mbox{even}, -1).
\end{align}

Note that in this setting, we also have
\[
\dim T^* \F_{\mbf d + (1,\cdots, 1),\ve} -\dim T^* \F_{\mbf d,\ve} = \dim \mathcal O_{i_{\ve}(0)} =
\frac{1}{2} ( (2v+n)(n-1) -\ve (n - \delta_{1,\bar n})).
\]

By an argument similar to the previous sections, we have the following stabilization result. 

\begin{prop}
\label{i-homo}
The transfer maps $\phi_{v+n,v}$ defined in (\ref{phi-n}) and (\ref{phi-n-2}) satisfy  
\[
\phi_{v+n, v} ( [ \overline Z_{A+ I,\ve}] ) = 
\begin{cases}
[\overline Z_{A,\ve}] & \mbox{if} \  A\in \Theta_{v}, \\
0 & \mbox{if}\ A\not\in \Theta_v.
\end{cases}
\] 
\end{prop}

\section{A Lagrangian construction of $\U(\mathfrak{sl}^{\theta}_n)$ and its rational modules}

With Propositions~\ref{stab} and~\ref{i-homo} in hand, we are ready to present a geometric realization of  $ \U(\mathfrak{sl}^{\theta}_n)$ and its Lusztig form
inside the limit  
$\varprojlim_{k, v} \H(Z_{n, v+2kn,\ve})$ of the projective system $(\H(Z_{n, v,\ve}), \phi_{v+2n, v})$ of algebras  from (\ref{H-top}) and (\ref{phi-top}), where
$v\in I_{n,\ve}$. 
This is done depending on the parity of $n$. 
It is similar to a quantum version established in~\cite{LW18}.

We also show that the Lusztig form admits a basis and that the modules arising from geometry are rational. 

\subsection{Odd case}

In this section, we assume that $n$ is odd.
We consider the following elements in the projective limit  
$\varprojlim_{k, v} \H(Z_{n, v+2kn,\ve})$:
$$
[\overline Z_{\overline A,\ve}]:= \varprojlim_{k} [\overline Z_{A+2kI,\ve}], 
\quad \overline A\in \overline  \Theta_{n, v  +\infty}.
$$
In light of Remark~\ref{Z-infty}, it can be regarded as the fundamental class of the $\overline Z_{\overline A,\ve}$
in the `top-Borel-Moore homology' of $Z_{n, \infty,\ve}$.
Recall $r=\lfloor n/2\rfloor$. For all $1\leq i\leq n-1$, we define 
\begin{align}
\label{generators}
\begin{split}
\begin{cases}
&\e_{i,\ve} = \sum_{\overline A} 
(-1)^{a_{ii}+\delta_{i,r+1} ( \delta_{-1,(-1)^v}  + 1 -\delta_{1,\ve})}
[\overline Z_{\overline A,\ve}],\\
& \f_{i,\ve}  = \sum_{\overline A}  (-1)^{a_{i+1, i+1} +\delta_{i,r} ( \delta_{-1,(-1)^v}  + 1 -\delta_{1,\ve}) }
[\overline Z_{\overline A,\ve}],\\
& \h_{i,\ve}  = \sum_{\overline A} (a_{ii} - a_{i+1, i+1} )
[\overline Z_{\overline A,\ve}],\\
& \mbf 1_{\overline \lambda,\ve}  = [\overline Z_{\overline A,\ve}],
\end{cases}
\end{split}
\end{align}
where the first (resp. second) sum runs over all $\overline A$ such that $A- E^{\theta}_{i,i+1}$ (resp. $A- E^{\theta}_{i+1,i}$), 
with $E^{\theta}_{i, i+1} = E_{i, i+1} + E_{n+1-i,n-i}$, is diagonal, the third sum runs over all diagonals, and the $\overline A$ in the last equality is diagonal such that $\ro(A) = \lambda$.
Note that the sums are infinite and are well-defined in $\varprojlim_{k, v} \H(Z_{n, v+2kn,\ve})$ thanks to Proposition~\ref{stab}.
By definition, one has 
\[
\e_{i,\ve} =\f_{\theta(i),\ve}\ \mbox{and}\ \h_{i,\ve} +\h_{\theta(i),\ve}=0.
\]
Let 
$\bU_{n,\ve}$ be the subalgebra of $\varprojlim_{k, v} \H(Z_{n, v+2kn,\ve})$ generated by
$\e_{i,\ve}$ and $\h_{i,\ve}$ for $1\leq i\leq n-1$. 
Let $\dot \bU_{n,\ve}$ be the $\bU_{n,\ve}$-bimodule generated by $\mbf 1_{\overline \lambda,\ve}$ for all $\overline \lambda\in \overline  \Lambda_{v+\infty}$ 
where $v\in I_{n,\ve}$, where $I_{n,\ve}$ is from (\ref{I}). 
Clearly,
\begin{align}
\begin{split}
\mbf 1_{\bar \lambda,\ve}\mbf  1_{\bar \mu,\ve} & =\delta_{\bar \lambda,\bar \mu} \mbf 1_{\bar \lambda, \ve}, \ \forall \bar \lambda, \bar \mu.\\
\e_{i,\ve} \mbf 1_{\bar\lambda,\ve} & =  \mbf 1_{\bar \lambda + \bar \delta_i - \bar \delta_{i+1} - \bar \delta_{\theta(i) } + \bar \delta_{\theta(i) + 1} , \ve} \e_{i,\ve},\  \forall i, \bar \lambda.
\end{split}
\end{align}
So $\dot \bU_{n,\ve}$ is naturally an associative algebra without unit.

\begin{thm}
\label{U-iso-odd}
Assume that $n$ is odd. There is a unique algebra isomorphism 
\[\U(\mathfrak{sl}^{\theta}_n) \to \bU_{n,\ve}, \
e_{i,\theta} \mapsto \e_{i,\ve}, h_{i,\theta}\mapsto \h_{i,\ve}, \quad \forall 1\leq i\leq n-1.
\]
Moreover,
there is  an algebra isomorphism 
$\dot \U(\mathfrak{sl}^{\theta}_n) \to \dot \bU_{n,\ve}$ defined by
\[
e_{i,\theta}^\ell 1_{\bar \lambda} \mapsto
\e^\ell_{i,\ve} \mbf 1_{\bar \lambda,\ve},
\quad \forall \ell=0,1.
\]
\end{thm}

The proof will be  given in Section~\ref{Schur}.

\begin{rem}
(1). When $\ve=1$, the notation $\H_{n,\ve}$   depends on the parity of $v$. But they are isomorphic via the transfer map $\phi_{v+n,v}$ in (\ref{phi-n}).

(2). The generators $\e_i$ and  $\f_i$ here correspond to $\f_i$ and $\e_i$ in~\cite{BKLW} respectively.
\end{rem}

\subsection{Even case}

In this section, we assume that $n$ is even. We define the algebras $\bU_{n,\ve}$ and $\dot \bU_{n,\ve}$ in exactly the same manner as in the odd case, except for the generator $\e_{r,\ve}=\f_{r,\ve}$ for $r=n/2$: precisely they are given by
\begin{align}
\begin{split}
\e_{i,\ve} &= \sum_{\overline A} 
(-1)^{a_{ii}- \delta_{i, r} \delta_{1, \ve}}
\left ([\overline Z_{\overline A,\ve}]+ \delta_{i, r} (-1)^{\delta_{1,\ve}} \delta_{-1, (-1)^{a_{ii}}}[Z_{\overline A_r,\ve}] \right ),\ \forall 1\leq i\leq n-1,
%
\end{split}
\end{align}
where the sum runs over all $\overline A$ such that $A- E_{i,i+1}^{\theta}$ is diagonal and $A_r= A-E_{r,r+1}^{\theta} +E_{r,r}^{\theta}$. Note that $\h_{r,\ve}=0$.

\begin{thm}
\label{U-iso-even}
The statements in Theorem~\ref{U-iso-odd} remain true for  $n$ even. 
\end{thm}

The proof will be given in Section~\ref{Schur}.

\subsection{Homological bases}

In this section, we show that $\dot \bU_{n,\ve}$ admits a basis $[\overline Z_{\overline A,\ve}]$ for $\overline A\in \overline \Theta_{v+\infty}$ for various $v$. 
This result provides a key step to the proof of  Theorems~\ref{U-iso-odd} and~\ref{U-iso-even}.

Recall $\Theta_{v+\infty}$ from (\ref{Theta}). We define a partial order $\preceq$ on $\Theta_{v+\infty}$ by 
\[
A\preceq B \ \mbox{if and only if} \ \sum_{r\leq i; s\geq j} a_{rs} \leq \sum_{r\leq i; s\geq j} b_{rs}, \quad \forall i< j. 
\]
Note that the definition does not involve any diagonal entries, hence this partial order is also well-defined on
$\overline \Theta_{v+\infty}$.
We further define a refinement $\sqsubseteq$ by
\[
A\sqsubseteq B \ \mbox{if and only if} \ A\preceq B, \ro(A)=\ro (B), \co(A) =\co(B).
\]

Recall $E^{\theta}_{h,h+1} = E_{h,h+1} + E_{n+1-h, n-h}$. 
For $h, j_0\in [1,n-1]$, $c\in \mbb N$ and $A\in \Theta_{v+\infty}$,
we set
\[
A_{h, c;j_0} = A - c (E^{\theta}_{h, j_0} - E^{\theta}_{h+1, j_0}).
\]

\begin{lem}
\label{H-leading}
Let $C, A\in  \Theta_{v+\infty}$. Assume that
$\co (C) = \ro (A)$ and $C - c E^{\theta}_{h, h+1}$ is diagonal for some $c\in \mbb N$, $h\in [1,n-1]$. 
Assume further there is an integer $j_0$
such that 
\[
\begin{cases}
j_0\in [1, n] &\mbox{if}\  h\neq r, r+1,\\  
j_0\in [1, r) & \mbox{if}\  h= r,\\
 j_0\in (r, n] & \mbox{if}\ h=r+1, 
\end{cases}
\ \mbox{and} \
\begin{cases}
a_{h, j} = 0 & \forall 1\leq j \leq j_0-1,\\
a_{h, j_0} = c & \\
a_{h+1, j} =0 & \forall 1\leq j \leq j_0.
\end{cases}
\]
Then we have
\begin{align}
[\overline Z_{\overline C,\ve}] * [\overline Z_{\overline A,\ve}] = [\overline Z_{\overline A_{h,c; j_0},\ve} ]  + \mbox{lower terms},
\end{align}
where ``lower terms'' stands for a finite sum of fundamental class $[\overline Z_{\overline B, \ve}]$ such that $\overline B \sqsubseteq \overline A_{h,c; j_0}$. 
\end{lem}

\begin{proof}
This is due to~\cite[Lem. 3.9]{BKLW}, ~\cite[Thm. 2.7.26]{CG97} and Proposition~\ref{stab}.
\end{proof}

We define a total order on the set $\{(i, j) | 1\leq j < i \leq n\}$ by setting $(i, j)\triangleleft (k,l)$ if and only if $i-j >k-l$ or $i-j=k-l$ and $i<k$. 
For all $A\in \Theta_{v+\infty}$ with $a_{ii}>>0$, 
we define inductively with respect to the order $\triangleleft$ the  $\frac{n(n-1)}{2}$ matrices $G_{ij} \equiv G^A_{ij} \in \Theta_{v+\infty}$ for $i> j$ as follows.
$G_{n, 1}$ is defined by the condition $G_{n, 1} - a_{n, n-1} E^{\theta}_{n, n-1}$ is diagonal
and $\ro (G_{n, 1}) = \ro (A)$.  Put $G_{i-1, j-1} = G_{n, n-i}$ if $j=1$. 
Assume that $G_{i-1, j-1}$ is defined, we define $G_{ij}$ by the conditions that $\ro (G_{i, j}) = \co (G_{i-1, j-1})$ and
$G_{ij} - \sum_{k\geq i} a_{kj} E^{\theta}_{i, i+1}$ is diagonal.
In particular,  for all $2\leq i\leq n$, we have
\[
\ro (G_{n, 1}) = \ro (A), \co (G_{n, n-i}) =\ro (G_{i, 1}), \co (G_{ij}) = \ro (G_{i+1, j+1}) , \co (G_{n,n-1}) = \co (A).
\]
We set
\begin{align}
\label{H-m_A}
\begin{split}
m_{\overline A,\ve} 
= [\overline Z_{\overline G_{n, 1},\ve}]  *
([\overline Z_{\overline G_{n-1, 1},\ve}] *  [\overline Z_{\overline G_{n, 2},\ve}])  \cdots  
([\overline Z_{\overline G_{i, 1},\ve}]  * [\overline Z_{\overline G_{i+1, 2},\ve}]*   \cdots\\
\cdots *  [\overline Z_{\overline G_{n, n-i+1},\ve}]) 
  \cdots 
([\overline Z_{\overline G_{2,1},\ve}] \cdots [\overline Z_{\overline G_{n, n-1},\ve}]).
\end{split}
\end{align}
Note that the product is taken with respect to the order $\triangleleft$.
We have two bases for $\dot \bU_{n,\ve}$  from the above construction.

\begin{prop}
\label{H-estimate}
For all $A\in \Theta_{v+\infty}$, we have $m_{\overline A, \ve} = [\overline Z_{\overline A,\ve}] + \mbox{lower terms}$.
Moreover, 
the elements $[\overline Z_{\overline A,\ve} ]$ (resp. $m_{\overline A,\ve}$)  for all $\overline A\in\overline  \Theta_{v+\infty}$ and $v\in I_{n,\ve}$ 
form a basis in $\dot \bU_{n,\ve}$. 
\end{prop}

\begin{proof}
The first statement is obtained by applying Lemma~\ref{H-leading}.
So by an induction with respect to  $\sqsubseteq$, we have
$[\overline Z_{\overline A,\ve}] \in \dot \bU_{n, \ve}$. The $[\overline Z_{\overline A,\ve}]$s are clearly linearly independent. 
In light of Proposition~\ref{stab}, we see that any product of the fundamental classes
$[\overline Z_{\overline C,\ve}]$ such that  $C- c E^{\theta}_{i,i+1}$ is diagonal for some $c$ is a linear sum of
$[\overline Z_{\overline A,\ve}]$ for various $A$. 
This shows that the $[\overline Z_{\overline A,\ve}]$s form a basis of $\dot \bU_{n,\ve}$. So does the $m_{\overline A,\ve}$s. 
The proposition is thus proved.  
\end{proof}

\begin{rem}
By Proposition~\ref{H-estimate} and Remark~\ref{Z-infty}, $\dot \bU_{n, \ve}=\mrm{span}_{\mbb Q}\{ [\overline Z_{\overline A,\ve}]| \overline A\in \overline \Theta_{\infty}\}$ and can be thought of as the `top Borel-Moore homology' of $Z_{n, \infty,\ve}$.
\end{rem}

\subsection{Rational modules}

We provide a geometric realization of rational simple modules. 
We also identify the category $\CC(\mathfrak{sl}^{\theta}_n)$ with certain category of $\G_{v+\infty,\ve}$-equivariant perverse sheaves on the ind-variety $\N_{n, v+\infty,\ve}$.

Recall that an orbit $\mathcal O_{\mu,\ve}$, or simply $\mu$, in $\N_{n, v, \ve}$  is relevant if  for any $x\in \mathcal O_{\mu,\ve}$,
\[
\dim \pi^{-1}(x) \cap \F_{\mbf d,\ve} = \frac{1}{2} \mrm{codim}_{\N_{n, v, \ve}} \mathcal O_{\mu,\ve}, \quad \mbox{for some}\ \mbf d\in \Lambda_v. 
\]
Let $\mbf C(\mu )= \G_{v,\ve}(x)/\G^o_{v,\ve}(x)$ for any $x\in \mathcal O_{\mu,\ve}$.
Note that $\mbf C(\mu)$ is independent of the choice of $x$. 
The finite group $\mbf C(\mu)$ is known to be a product of copies of $\mbb Z/2\mbb Z$.  
To each representation $\psi\in \mbf C(\mu)$, one can attach a $\G_{v,\ve}$-equivariant local system $\mathcal L_{\psi}$ on 
$\mathcal O_{\mu,\ve}$. 
Let $\IC_{\mu,\psi}$ denote the simple $\G_{v,\ve}$-equivariant  perverse sheaf on $\N_{n, v, \ve}$ so that  its support is 
$\overline{\mathcal O}_{\mu,\ve}$ and 
$\IC_{\mu,\psi}|_{\mathcal O_{\mu,\ve}}= \mathcal L_{\psi}[\dim \mathcal O_{\mu,\ve}]$;
see~\cite[1.6]{Lu07} for details.
If the local system $\psi$ is trivial, we simply write $\IC_\mu$ or $\IC_{\overline{\mathcal O}_{\mu,\ve}}$ instead. 
Recall the complex $L_{n, v, \ve}$ from (\ref{L}). 
Since there is a canonical $\G_{v,\ve}$-equivariant structure on $\IC_{\M_{n, v,\ve}}$ and $\pi$ is $\G_{v,\ve}$-equivariant,
the semisimple complex $L_{n,v,\ve}$ is $\G_{v,\ve}$-equivariant. 
Let $W_{\mu,\psi} = \Hom_{\G_{v,\ve}} (L_{n, v, \ve}, \IC_{\mu,\psi})$,
where the `$\Hom$' is taken in the category of $\G_{v,\ve}$-equivariant perverse sheaves on $\N_{n, v, \ve}$. 
Thus, as $\G_{v,\ve}$-equivariant perverse sheaves,  we can write
\begin{align}
\label{L-W}
L_{n, v, \ve} \cong \oplus_{(\mu,\psi)} \IC_{\mu,\psi} \otimes W_{\mu,\psi},
\end{align}
where the direct sum is over all pairs $(\mu,\psi)$ where
$\mu$ is a  relevant partition   and
$\psi$ runs over the set $\mbf C(\mu)^{\vee}$ of  irreducible representations of $\mbf C(\mu)$ such that $W_{\mu,\psi} \neq 0$.

Note that $\mbf C(\mu)\leq \mbf C(n^2, \mu)$.

\begin{prop}
\label{IC-stab}
We have $\mbf C(\mu)^{\vee} =\mbf C(n^2,\mu)^{\vee}$. 
Moreover,  with $a_{\ve}=\dim \mathcal O_{i_{\ve}(0)}$,
\begin{align}
i^*_{\ve} [-a_{\ve}] (\IC_{\nu,\psi}) = 
\begin{cases}
\IC_{\mu, \psi}, &\mbox{if}\ \nu = (n^2,\mu),\\
0, &\mbox{o.w.} 
\end{cases}
\end{align}
\end{prop}

\begin{proof}
If $\nu$ is not of the form $(n^2,\mu)$ for some $\mu$, then $\mathcal O_{\nu,\ve}\cap i_{\ve}( \N_{n, v, \ve}) =\O$, 
then $\overline{\mathcal O}_{\nu,\ve} \cap i_{\ve}( \N_{n, v,\ve})=\O$, and so
$i^*_{\ve}(\IC_{\nu,\psi}) =0$. 
By Proposition~\ref{phi-U}, we see that
\begin{align}
\label{L-U}
L_{n, v+2n, \ve}|_{U} = \IC_{\mathcal O_{i_{\ve}(0)}\cap U} \boxtimes L_{n, v,\ve}. 
\end{align}
The remaining statements follow from the above. 
\end{proof}

\begin{prop}
\label{V-stab}
For any relevant partition  $\mu$ and $\psi\in \mbf C(\mu)$, we have
\begin{align}
\label{W-stab}
W_{\mu,\psi}\cong W_{(n^2,\mu), \psi},
\end{align}
as irreducible  $\mathfrak{sl}^{\theta}_n$-modules.
\end{prop}

\begin{proof}
The equality (\ref{L-U}) implies that $W_{\mu,\psi} \cong W_{(n^2,\mu), \psi}$.
Moreover  the following diagram is commutative in light of Proposition~\ref{IC-stab}. 
\[
\begin{CD}
\End_{\G_{v+2n,\ve}} (L_{n, v+2n, \ve}) \times \Hom_{\G_{v+2n,\ve}} (L_{n, v+2n, \ve}, \IC_{(n^2,\mu),\psi} ) @>>> \Hom_{\G_{v+2n,\ve}} (L_{n, v+2n, \ve}, \IC_{(n^2,\mu), \psi}) \\
@Vi^*_{\ve}[-a_{\ve}]\times i^*_{\ve}[-a_{\ve}] VV @VV i^*_{\ve}[-a_{\ve}] V\\
\End_{\G_{v,\ve}} (L_{n, v, \ve}) \times \Hom_{\G_{v,\ve}} (L_{n, v,\ve}, \IC_{\mu, \psi}) @>>> \Hom_{\G_{v,\ve}}(L_{n, v, \ve}, \IC_{\mu, \psi}).
\end{CD}
\]
But  we have 
\begin{align}
\label{Z-L}
\H(Z_{n, v, \ve})\cong \End_{\G_{v,\ve}} (L_{n, v, \ve})
\end{align} 
in light of (\ref{Z-G-fix}) and~\cite{Lu88} (see Remark~\ref{remark} (2) for an alternative proof).
Therefore the isomorphism (\ref{W-stab}) is compatible with the $\mathfrak{sl}_n^{\theta}$-actions. 
Proposition follows immediately.
\end{proof}

Observe that $W_{\mu,\psi}$ is a weight module: 
\[
W_{\mu, \psi}=\oplus_{\mbf d} W_{\mu,\psi}(\mbf d), \quad W_{\mu,\psi}(\mbf d) =W_{\mu,\psi} \cap \H_*(\pi^{-1}(x_{\mu})\cap \F_{\mbf d,\ve}).
\]
Moreover, the $e_{i,\ve}$-action is  locally nilpotent for all $i$. 
Recall the category $\CC(\mathfrak{sl}_n^{\theta})$ and $L'(\omega, \omega')$ from Section~\ref{Rep}.

\begin{thm}
\label{Springer}
The list of $W_{\mu,\psi}$,  where $\mu$ is a partition such that $\mu_1\leq n$ and $\mu_2<n$
and $\psi\in \mbf C(\mu)^{\vee}$,
exhausts all rational simple modules  $L'(\omega, \omega')$  in $\CC(\mathfrak{sl}_n^{\theta})$.
\end{thm}

\begin{proof}
We introduce the following Grothendieck resolution. 
\[
\pi_g:  \tilde{\mathfrak{g}}_{n,v,\ve} \to \mathfrak{g}_{v,\ve}, \quad \mathfrak{g}_{v,\ve}=\mrm{Lie} \G_{v,\ve},
\]
where
\[
 \tilde{\mathfrak{g}}_{n,v,\ve} =\{ (x, F)\in \mathfrak g_{v,\ve}\times \F_{n, v, \ve} | x(F_i) \subseteq F_i,\ \forall i\},
\]
and $\pi_g$ is the first projection.  
Let $L^g_{n,v,\ve} = (\pi_g)_* \IC_{\tilde{\mathfrak{g}}_{n,v,\ve}}$.
It is well-known (e.g.,~\cite{G98}) that 
\begin{align}
\label{Fourier-L}
\mbb F(L^g_{n, v, \ve}) = L_{n, v, \ve}
\end{align}
where  $\mbb F$ is the Fourier transform. Let $\mathfrak g^{rs}_{v,\ve}$ be the open dense subvariety of $\mathfrak g_{v,\ve}$ consisting of all regular semisimple elements. 
Let $\pi_1(\mathfrak g^{rs}_{v,\ve})$ be its fundamental group with respect to a base point, say $x$. 
Observe that the fiber of $x$ under $\pi_g$ is naturally parametrized by the collection of  matrices  $T=(t_{ij})_{1\leq i\leq n, 1\leq j\leq v}$, which has a 1 in each column and 0 else where,  subject to the symmetry 
$t_{ij} = t_{n+1-i, v+1-j}$. 
In particular, we see that the stalk of $L^g_{n, v, \ve}$ at $x$ is isomorphic to the tensor space $(\mbb Q^n)^{\otimes d}$
Now assume that $(v,\ve)\neq (\mbox{even}, 1)$, applying the argument in~\cite[Section 10]{G98} (see also \cite{BG99}), we have
\begin{align}
\label{Fourier}
\begin{split}
\oplus_{\mu, \phi} \mrm{End}(W_{\mu,\phi}) & \overset{(\ref{L-W})}{\cong} \mrm{End}_{\G_{v,\ve}}(L_{n,v,\ve}) \\
&\overset{(\ref{Fourier-L})}{\cong} \mrm{End}_{\G_{v,\ve}} ( L^g_{n, v, \ve}) \\
& \cong \mrm{End} (L^g_{n, v,\ve}|_{\mathfrak g_{v,\ve}^{rs}}) \quad \quad \quad (\mbox{perverse extension})\\
& \cong \mrm{End}_{\pi_1(\mathfrak g_{v,\ve}^{rs})} ( L^g_{n, v,\ve}|_{x}) \ \quad (\mbox{monodromy}) \\
&\cong \mrm{End}_{W_{B_d}} ( (\mbb Q^n )^{\otimes d}) \quad \quad (\mbox{\cite[Lem. 10.2]{G98}})
\end{split}
\end{align}
where $W_{B_d}$ is the Weyl group of type $B$ of rank $d=\lfloor v/2\rfloor$.
If $(v,\ve) =(\mbox{even}, 1)$, then $\G_{v,\ve} \cong \mrm{SO}_{v} \rtimes \mbb Z/2\mbb Z$ and
the above argument holds with the following modifications:
\begin{align}
\label{Fourier-2}
\begin{split}
\oplus_{\mu, \phi} \mrm{End}(W_{\mu,\phi}) & \cong \mrm{End}_{\G_{v,\ve}}(L_{n,v,\ve}) \\
& \cong \mrm{End}_{\mbb Z/2\mbb Z} (L^g_{n, v,\ve}|_{\mathfrak g_{v,\ve}^{rs}}) \\
& \cong \mrm{End}_{\pi_1(\mathfrak g_{v,\ve}^{rs})\rtimes \mbb Z/2\mbb Z} ( L^g_{n, v,\ve}|_{x}) \\
&\cong \mrm{End}_{W_{D_d}\rtimes \mbb Z/2\mbb Z} ( (\mbb Q^n )^{\otimes d}) \\
&\cong \mrm{End}_{W_{B_d}}  ( (\mbb Q^n )^{\otimes d})
\end{split}
\end{align}
where $W_{D_d}$ is the Weyl group of type $D$ of rank $d$.
Thanks to  ~\cite{LZ19, SS99} , 
the tensor space $(\mbb Q^n)^{\otimes d}$, as a natural $\mathfrak{sl}_n^{\theta}$-module,  is rational and so is  $W_{\mu,\phi}$, as its simple summand. 
Moreover, any $L'(\omega,\omega')$ is a simple summand in $(\mbb Q^n )^{\otimes d}$ for $d$ large enough.
This implies that $L'(\omega,\omega')$ has to be some $W_{\mu,\psi}$ thanks to Proposition~\ref{V-stab}. The proof is finished.  
\end{proof}

\begin{rem}
\label{remark}
(1). It is desirable to describe relevant partitions  and the associated set $\mbf C(\mu)^{\vee}$. 
It is also desirable to establish a correspondence  between $W_{\mu,\phi}$ and $L'(\omega,\omega')$.

(2). We provide an alternative proof of (\ref{Z-L}). That statement is well-known except the case: $(v,\ve)=(\mbox{even}, 1)$.
If $n$ is odd, we have the following commutative diagram
\[
\begin{CD}
\H(Z_{n,v+n,\ve}) @>\cong>> \End_{\G_{v+n,\ve}}(L_{n, v+n,\ve})\\
@V\phi_{v+n,v} VV @VV i^*_{\ve}[-a_{\ve}] V\\
\H(Z_{n,v,\ve})  @>\psi>> \End_{\G_{v,\ve}}(L_{n,v,\ve}).
\end{CD}
\]
Since $i^*_{\ve}[-a_{\ve}]$ is surjective by Proposition~\ref{IC-stab}, $\psi$ is surjective too. Thanks to  (\ref{Fourier-2}), we know that the dimensions of the domain and range of $\psi$ are the same. So
$\psi$ must be  an isomorphism.  The $n$=even case follows also because they are subalgebras of $n$=odd cases.

(3). The analyses (\ref{Fourier})-(\ref{Fourier-2}) also provides an alternative proof that $\H(Z_{n, v,\ve})$ receives a surjective homomorphism from $\U(\mathfrak{sl}_n^{\theta})$.
\end{rem}

Let $\Q_{n, v,\ve}$ be the semisimple abelian subcategory of the category of perverse sheaves on $\N_{n, v, \ve}$ whose simple objects are 
direct summands in $L_{n,v,\ve}$. 
Then 
$
P\mapsto \Hom_{\G_{v,\ve}} (L_{n, v, \ve}, P)
$
defines an equivalence of categories
\[
\Q_{n, v, \ve} \to \H (Z_{n, v, \ve})\mbox{-mod}.
\]

Consider the projective system $(\Q_{n, v, \ve}, i^*_{\ve}[-a_{\ve}]))_{v}$. In view of Proposition~\ref{IC-stab}, we can consider the full 
abelian subcategory $\Q_{n, \infty, \ve}$ in the limit of the above  projective system
whose simple objects are $\varprojlim_\ell (i^*_{\ve}[-a_{\ve}])^\ell ( \IC_{x, \psi})$. 
Theorem~\ref{Springer} can be reformulated as follows (compare~\cite{VV03}).

\begin{prop}
We have an equivalence of abelian categories
\begin{align}
\Q_{n, \infty,\ve} \cong \CC( \mathfrak{sl}_n^{\theta}).
\end{align}
\end{prop}

\section{Proof of Theorems~\ref{U-iso-odd} and~\ref{U-iso-even}}

This section is devoted to the proof of the main Theorems~\ref{U-iso-odd} and~\ref{U-iso-even}. 
It takes three steps. The first step is to prove Theorem~\ref{U-iso-odd} for $n=3$ by constructing a series of representations naturally arising from geometry and show that the algebras involved acts
faithfully on the product of all these representations.
The second step is to check that the defining relations of $\U(\mathfrak{sl}_n^{\theta})$ are satisfied on the Schur algebra level. 
The theorem will then follow thanks to Proposition~\ref{H-estimate} for $n$ odd finally. 
Theorem~\ref{U-iso-even} is a consequence of Theorem~\ref{U-iso-odd} because the defining relations in (\ref{Def-even})
can be deduced from those in (\ref{Def-odd}).

As a digression, we discuss the $(n, v, \ve)= ( even, odd, 1)$  case and a natural representation.

\subsection{Analysis in the $n=3$ case}
\label{n=3}

Recall that $V$ is a $v$-dimensional complex vector space with an $\ve$-form.
In this section, we set
\[
d\equiv d_v:=\lfloor v/2\rfloor.
\]
Let $\Gr_i^v$ (resp. $\Gr_{i,\ve}^v$) be the (resp. isotropic) Grassmannian of $i$-dimensional (resp. isotropic) subspaces in $V$. 
Clearly $\Gr_{i,\ve}^v$ is empty unless $0\leq i\leq d$.
We set
\begin{align}
Y^+_{i,\ve} =\{ (F, F') \in \Gr_{i+1,\ve}^v \times \Gr_{i,\ve}^{v} | F \overset{1}{\supset} F'\},\
Y^-_{i,\ve} = \{ (F, F')\in \Gr_{i-1,\ve}^v \times \Gr_{i,\ve}^v | F\overset{1}{\subset} F'\}. 
\end{align}
Here the notations ``$\overset{1}{\supset}$'' and ``$\overset{1}{\subset}$'' denote that the vector spaces involved are contained in each other in the obvious way and 
the dimension difference is $1$.

Consider the fundamental classes $[T^*_{Y^{+}_{i, \ve}} ( \Gr_{i+1,\ve}^v \times \Gr_{i,\ve}^v)]$
(resp. $[T^*_{Y^-_{i,\ve}} (\Gr_{i-1,\ve}^v \times \Gr_{i,\ve}^v)]$)
of the conormal bundle of $Y^+_{i,\ve}$ (resp. $Y^-_{i, \ve}$)  in $\Gr_{i+1,\ve}^v \times \Gr_{i,\ve}^v$
 in $\H_*(Z_{n,v,\ve})$.
The following lemma is the analogue of~\cite[4.2.6]{CG97}.

\begin{lem}
We have
\begin{align}
[T^*_{Y^{+}_{i, \ve}} ( \Gr_{i+1,\ve}^v \times \Gr_{i,\ve}^v)] * [\Gr_{i,\ve}^v] & =(-1)^i  (i+1) [\Gr_{i+1,\ve}^v],\\
[T^*_{Y^-_{i,\ve}} (\Gr_{i-1,\ve}^v \times \Gr_{i,\ve}^v)] *[\Gr_{i,\ve}^v] & = (-1)^{ v-1-\delta_{1,\ve}} 
 2(d - i+1  ) [\Gr_{i-1,\ve}^v].
\end{align}
\end{lem}

\begin{proof}
We shall apply  Theorem 2.7.26 in~\cite{CG97} to our situation. 
Let us prove the first formula. 
We set
$X_1= \Gr_{i+1,\ve}^v$,   $X_2= \Gr_{i,\ve}^v$ and $X_3=\{\bullet\}$  a point.
We also set
$Y_{12} = Y^+_{i,\ve}$
and $Y_{23} = \Gr_{i,\ve}^v. $
Let $p_{ij}$ be the projection of $X_1\times X_2\times X_3$ to the factor $X_i\times X_j$. 
Then 
\[
p_{12}^{-1} (Y_{12}) \cap p_{23}^{-1} (Y_{23}) = Y_{12}\times \{\bullet\} \subseteq X_1\times X_2\times X_3.
\] 
The restriction of $p_{13}$ to $p_{12}^{-1} (Y_{12}) \cap p_{23}^{-1} (Y_{23}) $ is the first projection
$
Y_{12} \to \Gr_{i+1,\ve}^v$, 
$(F, F') \mapsto F, $
which is a smooth fibration with fiber isomorphic to  $\Gr^{i+1}_i$. 
So applying Theorem 2.7.26 in~\cite{CG97} and Lemma 8.5 in~\cite{N98} to get
\[
[T^*_{Y^{+}_{i, \ve}} ( \Gr_{i+1,\ve}^v \times \Gr_{i,\ve}^v)] * [\Gr_{i,\ve}^v] = (-1)^{\dim \Gr^{i+1}_i} (i+1)\chi (\Gr^{i+1}_i) [\Gr_{i+1,\ve}^v] = (-1)^{\dim \Gr^{i+1}_i} (i+1) [\Gr_{i+1,\ve}^v].
\]
The first formula is proved.

For the second one, we set
$X_1 = \Gr_{i-1,\ve}^v$, $X_2 =\Gr_{i,\ve}^v$ and  $X_3=\{\bullet\}.$
We also set
$Y_{12} = Y^-_{i,\ve} $ and  $Y_{23} =\Gr_{i,\ve}^v.$
Then 
$
p_{12}^{-1} (Y_{12}) \cap p_{23}^{-1} (Y_{23}) = Y_{12}\times \{\bullet\} \subseteq X_1\times X_2\times X_3.
$
The restriction of $p_{13}$ to $p_{12}^{-1} (Y_{12}) \cap p_{23}^{-1} (Y_{23}) $ is the first projection
$
Y_{12} \to \Gr_{i-1,\ve}^v$, $(F, F') \mapsto F,
$
which is a smooth fibration with fiber at $F$ being $\{ F' \in \Gr_{i,\ve}^v| F\overset{1}{\subset} F'\}\cong \Gr_{1, \ve}^{ v-2(i-1)}$.
Recall 
\begin{align}
\dim \Gr_{1, \ve}^{ v} = v-1 -\delta_{1,\ve}
\quad \mbox{and}\quad
\chi (\Gr_{1,\ve}^v) = v- \delta_{-1,(-1)^v}=2d.
\end{align}
By applying again~\cite[Thm. 2.7.26]{CG97} and~\cite[Lem. 8.5]{N98}, we have
\begin{align*}
[T^*_{Y^-_{i,\ve}} (\Gr_{i-1,\ve}^v \times \Gr_{i,\ve}^v)] *[\Gr_{i,\ve}^v] & = (-1)^{\dim \Gr_{1, \ve}^{ v-2(i-1)}}
\chi ( \Gr_{1,\ve}^{v-2(i-1)}) [\Gr_{i-1,\ve}^v]\\
&  =  (-1)^{v- 2(i-1) -1 -\delta_{1,\ve}} (v- \delta_{-1, (-1)^v}  - 2(i-1) ) [\Gr_{i-1,\ve}^v]\\
&  =(-1)^{v-1-\delta_{1,\ve}}  2(d-i+1) [\Gr_{i-1,\ve}^v].
\end{align*}
The second formula is proved. This finishes the proof.
\end{proof}

Let  
\begin{align}
\label{phi}
\phi_v: \varprojlim_{k, v'} \H(Z_{n, v'+2kn,\ve}) \to \H(Z_{n, v,\ve})
\end{align}
be the natural map from the projective limit. 
We set
\begin{align*}
\e_{\ve} & := \phi_v(\e_{1,\ve}) = \sum_{i=0}^{d-1} (-1)^i [T^*_{Y^{+}_{i, \ve}} ( \Gr_{i+1,\ve}^v \times \Gr_{i,\ve}^v)], 
\\
\f_{\ve} & :=\phi_v(\f_{1,\ve}) = \sum_{i=1}^{d} (-1)^{ v-1- \delta_{1,\ve} }  [T^*_{Y^-_{i,\ve}} (\Gr_{i-1,\ve}^v \times \Gr_{i,\ve}^v)], \\
\h_{\ve} &: =\phi_v(\h_{1,\ve}) = \sum_{i=0}^d (3i-2d) [T^*_{\Delta} (\Gr_{i,\ve}^v\times \Gr_{i,\ve}^v)],
\end{align*}
where $\Delta$ stands for the diagonal of $\Gr_{i,\ve}^v\times \Gr_{i,\ve}^v$.
Then the formulas in the above lemma read 

\begin{align}
\label{rank-one-action}
\begin{split}
\e_{\ve}* [\Gr_{i,\ve}^v]  = (i+1) [\Gr^v_{i+1,\ve}], 
\
\f_{\ve}* [\Gr_{i,\ve}^v]  = 2(d- i +1) [\Gr^v_{i-1,\ve}],\\
\
\h_{\ve} *[\Gr_{i,\ve}^v]  = (3i-2d) [\Gr_{i,\ve}^v].
\end{split}
\end{align}

With the above formula, one can check immediately the following.

\begin{lem}
The action $\{ \e_{\ve}, \f_{\ve}, \h_{\ve}\}$ on the space
\[
W_{v,\ve} = \mrm{span}_{\mbb Q} \{ [\Gr_{i,\ve}^v ] | 0\leq i\leq d\}  \subseteq \oplus_{i=0}^d \H_* ( \Gr_{i,\ve}^v),
\]
defines an irreducible representation of $\U(\mathfrak{sl}_3^{\theta})$ and $\dot \U(\mathfrak{sl}_3^{\theta})$.
\end{lem}

We shall show later that this action factors through $\dot \bU_{3,\ve}$.
The following proposition will then be critical in checking the most nontrivial  nonhomogeneous Serre relations in $\U(\mathfrak{sl}_3^\theta)$.

\begin{prop}
\label{Faith}
The algebra $\dot \U(\mathfrak{sl}_3^{\theta})$ acts faithfully on $\oplus_{v\geq 0} W_{v,\ve}$.
\end{prop}

\begin{proof}
To ease the burden of notations, we write 
$e_{\theta}=e_{1,\theta}$ and $ f_{\theta}= f_{1,\theta}.$
Let
\[
m_{a,b,c; \overline \lambda} = f^{a}_{\theta} \cdot  e^{b}_{\theta}  \cdot f_{\theta}^{c}  \cdot 1_{\overline \lambda}, 
\quad \forall a, b, c\in \mbb N, \overline \lambda\in \overline \Lambda_{v+\infty}, \ \mbox{for $v\in I_{3,\ve}$ 
}.
\]
It is known (e.g.,~\cite{LZ19}) 
that the set
\[
\{ m_{a,b,c;\overline \lambda} | a,b,c\in \mbb N, a \leq b,  \overline\lambda \in \overline \Lambda_{v+\infty}, \mbox{for $v \in I_{3,\ve}$}
\}
\]
is a basis of $\dot \U(\mathfrak{sl}_3^{\theta})$. Recall $ d_v=\lfloor v/2\rfloor$.
If $|\overline \lambda | = v$ mod $6$, then,  by (\ref{rank-one-action}),  we have
\begin{align}
\label{m-action}
m_{a,b,c; \overline \lambda} [ \Gr_{y,\ve}^v] & = P_{a, b, c}(y,d_v) [\Gr_{y-a+b-c, \ve}^v], \ \mbox{where}\\
\label{P(y,v)}
P_{a,b,c}(y,d_v) & =2^{a+c} \prod_{l=1}^a ( d_v- ( y- c+b-l)) \prod_{k=1}^b (y-c+k)  \prod_{j=1}^c (d_v- (y-j) ) .
\end{align}

For the sake of contradiction, we assume that there exist
$\overline \lambda \in \overline \Lambda_{v+\infty}$, $\{(a_i,b_i,c_i)\}_{i=1}^N$ and $C_i\in \mbb Q-\{0\}$, $m'\in \mbb N$ such that
$a_i-b_i + c_i=m'$ and, as an operator, 
\begin{align}
\sum_{i=1}^N C_i m_{a_i,b_i, c_i; \overline \lambda}=0 \in \prod_{v\geq 0} \End(W_{v,\ve}).
\end{align}
In light of (\ref{m-action}),
for any $v$ such that $v= |\overline \lambda|$ mod $6$, we have
\begin{align}
\label{Faith-2}
\sum_{i=1}^N C_i P_{a_i, b_i, c_i} (y, d_v)=0,\quad \forall 0\leq y\leq d_v.
\end{align}
We set
$
c^0=\mrm{min} \{ c_i| 1\leq i\leq N\}.
$
By (\ref{P(y,v)}) and the fact that $a_i\leq b_i$ for all $i$, there is 
\[
P_{a_i,b_i,c_i} (c^0, d_v) =0, \forall c_i> c^0, P_{a_i, b_i, c_i}(c^0,d_v) \neq 0,\forall c_i=c^0, \quad \mbox{if}\ v>>0.
\]
So,
when evaluating at $y=c^0$, 
the equation (\ref{Faith-2}) yields
\[
\sum_{i: c_i=c^0} C_i P_{a_i, b_i, c_i} (c^0,d_v)=0, \quad \forall v>>0.
\]
As a polynomial in $d_v$, the degree of $P_{a_i, b_i, c_i}(c^0, d_v)$ is $a_i+c_i$, and hence the left-hand side of the above  is a zero polynomial in $d_v$.
Set $p=\max \{a_i+c_i  | c_i =c^0\}$ and we must have
\[
\sum_{i, c_i=c^0, a_i+c_i = p } C_i P_{a_i, b_i, c_i}(c^0, d_v)|_{d_v^{m'}} =0,
\]
where $P_{a_i, b_i, c_i}(c^0,d_v)|_{d_v^{m'}}$ is the leading term of $P_{a_i, b_i, c_i}(c^0,d_v)$.
But there is only one triple $(a_{i'}, b_{i'}, c_{i'})$, that is $(p -c^0, p-m', c^0)$, subject to the conditions:
\[
c_{i'} =c^0, a_{i'} +c_{i'}= p, a_{i'}-b_{i'}+c_{i'}=m'.
\]
So $C_{i'} P_{a_{i'}, b_{i'}, c_{i'}}(c^0,d_v)|_{d_v^{m'}}=0$, which implies that 
$C_{i'}=0$ because $P_{a_{i'}, b_{i'}, c_{i'}}(c^0,d_v) \neq 0$. This contradicts with the assumption that $C_i\neq 0, \forall i$. 
The proposition is thus proved.
\end{proof}


We are ready to show Theorem~\ref{U-iso-odd} for $n=3$. 

\begin{prop}
\label{Rank-one-iso}
The assignments
$e_{\theta}^{a}1_{\overline \lambda}\mapsto \e_{\ve}^{a} \mbf 1_{\overline \lambda}$ and 
$f_{\theta}^{a}1_{\overline \lambda}\mapsto \f_{\ve}^{a} \mbf 1_{\overline \lambda}$, 
for $a=0, 1$, define an isomorphism 
$
\dot \U(\mathfrak{sl}_3^{\theta}) \to \dot \bU_{3,\ve}. 
$
Moreover, there is an isomorphism $\U(\mathfrak{sl}_3^{\theta}) \cong \bU_{3,\ve}$ defined by
$e_{\theta}\mapsto \e_{\ve}, f_{\theta}\mapsto \f_{\ve}$ and $h_{\theta}\mapsto \h_{\ve}$.
\end{prop}

\begin{proof}
By the universal property of projective limit, we have an algebra homomorphism
\[
\dot \bU_{3,\ve} \to \prod_{v>0} \End (W_{v,\ve}).
\]
By Proposition~\ref{Faith}, the above homomorphism factors through a surjective  homomorphism 
\[
\dot \bU_{3,\ve} \to \dot \U(\mathfrak{sl}_3^{\theta}).
\]
Now both sides have a basis in the name of $m_{\overline A,\ve}$ or $m_{\overline A}$ sending to each other, and therefore the above homomorphism must be an isomorphism. 

It is known that $\U(\mathfrak{sl}_3^{\theta})$ (resp. $\bU_{3,\ve}$) acts faithfully on $\dot \U(\mathfrak{sl}_3^{\theta})$ (resp. $\dot \bU_{3,\ve}$) and these actions are compatible under the isomorphism $\dot \U(\mathfrak{sl}_3^{\theta}) \cong \dot \bU_{3,\ve}$. So $\U(\mathfrak{sl}_3^{\theta})$ and $\bU_{3,\ve}$ must be isomorphic as well. 
The proposition is thus proved. 
\end{proof}

As an immediate consequence of the above proposition, we have  the following corollary, which is required in the proof of the  general case.

\begin{cor}
The elements $\e_{\ve}, \f_{\ve}, \h_{\ve}$  satisfy the relations
\begin{align}
\label{rank-one-Serre}
\begin{split}
[\h_{\ve},\h_{\ve}]& =0,\\
[\h_{\ve}, \e_{\ve}] = 3 \e_{\ve}, [\h_{\ve},\f_{\ve}]&=-3\f_{\ve},\\
\e_{\ve}^2 \f_{\ve} - 2 \e_{\ve} \f_{\ve} \e_{\ve} + \f_{\ve} \e_{\ve}^2 &= - 4 \e_{\ve}, \\
\f_{\ve}^2 \e_{\ve} - 2 \f_{\ve} \e_{\ve} \f_{\ve} + \e_{\ve} \f_{\ve}^2 &= - 4 \f_{\ve}.
\end{split}
\end{align}
\end{cor}

\begin{rem}
The terms $\f_{\ve}^2\e_{\ve}$, $\f_{\ve}\e_{\ve}\f_{\ve}$ and $\e_{\ve}\f_{\ve}^2$,
when restricted to $T^*(\Gr^v_{i-1, \ve}\times \Gr^v_{i,\ve})$, are linear combinations of three fundamental classes
$[\overline Z_{A_i,\ve}]$ for $i=0, 1, 2$ where
$
A_0 - E^{\theta}_{2,1},
A_1- 2 E^{\theta}_{2,1} - E^{\theta}_{3,2},
A_2- E^{\theta}_{2,1} - E^{\theta}_{3,1}
$
are diagonal. 
The proof of~\cite[(4.3.7)]{CG97} does not apply here to check the nonhomogeneous Serre relation. 
\end{rem}

\subsection{Analysis in the $n=2$ case}

We now carry out the calculus in the $n=2$ case. We observe that $\H(Z_{2, v,\ve})$ is naturally a subalgebra of $\H(Z_{3,v,\ve})$ by adding an extra flag $F_r$ in the flag $F$ involved. 
We set
\begin{align}
\label{t-d}
\t_{\ve,d} = (\e_{\ve} \f_{\ve} - \h_{\ve}) * [T^*_{\Delta}(\Gr^v_{d,\ve}\times \Gr^v_{d,\ve})] \in \H(Z_{2,v,\ve}),
\end{align}
where the multiplication is done in $\H(Z_{3,v,\ve})$.
Thanks to (\ref{rank-one-action}), the action of $\t_{\ve,d}$ on $[\Gr^{v}_{d,\ve}]$ is given by
\begin{align}
\label{t-act}
\t_{\ve, d} * [\Gr^v_{d,\ve}] = d[\Gr^v_{d,\ve}]. 
\end{align}

Moreover, it satisfies the following relation. 

\begin{lem}
\label{defining-t}
We have
$
(\t_{\ve, d} - d) (\t_{\ve, d} - d+2) \cdots (\t_{\ve, d} + d-2) (\t_{\ve, d} +d) =0. 
$
\end{lem}

\begin{proof}
An induction  argument yields
$
\e_{\ve} \f_{\ve}^n = n ( \f_{\ve}^{n-1} \e_{\ve} \f_{\ve} - 2 (n-1) \f_{\ve}^{n-1} ) - (n-1) \f_{\ve}^n \e_{\ve}
$
in $\H(Z_{3,v,\ve})$.
From the above formula and a simple induction, we have the following equality.
\[
\frac{1}{n!} \e^n_{\ve} \f^n_{\ve} * [T^*_{\Delta}(\Gr^v_{d,\ve}\times \Gr^v_{d,\ve})] = 
(\t_{\ve, d} +d) (\t_{\ve, d} +d -2) \cdots (\t_{\ve, d} + d - 2(n-1)). 
\]
Setting $n=d+1$, we see that the  left and, hence right, hand sides must be zero, and it provides the desired result. 
The lemma is proved.
\end{proof}

\begin{rem}
\label{t-defining}
The relation in Lemma~\ref{defining-t} is indeed the defining relation of the algebra generated by $\t_{\ve,d}$. 
See  also~\cite{LZ19}.
\end{rem}

Recall the algebra homomorphism $\phi_{v, v-4}: \H(Z_{2, v, \ve}) \to \H (Z_{2,v-4, \ve})$ for $n=2$ 
from  Section~\ref{BM-TM}.
The following proposition says that $\t_{\ve,d}$ behaves in the simplest possible way under the transfer map $\phi_{v, v-4}$.

\begin{prop}
\label{transfer-t}
We have
$\phi_{v,v-4} (\t_{\ve, d}) = \t_{\ve, d-2}$.
\end{prop}

\begin{proof}
Assume that $\ve=1$.
By considering the support of $\phi_{v, v-4}(\t_{\ve, d})$, we see that
$\phi_{v,v-4} (\t_{\ve, d}) = \t_{\ve, d-2} +a$ for some $a\in \mbb Q$.
But by Lemma~\ref{defining-t}, we must have $a=0$ or  $\pm 2$.

By Lemma~\ref{defining-t}, the finite dimensional representations of the algebra generated by $\t_{\ve,d}$ are one dimensional with scalars 
$-d, -d+2, \cdots, d-2, d$. If $a=\pm 2$, then 
of the algebra generated by $\t_{\ve,d-2}$ can be lifted to be 
pairwise non-isomorphic representations with 
scalars $ -(d-2) \pm 2, \cdots (d-2) \pm 2$ by Proposition~\ref{stab}.
Note that in this case, $\Gr^v_{d,\ve}$ is a union of two connected components of pure dimensional, say
$\Gr^{v,(i)}_{d,\ve}$ for $i=1,2$. The subspace spanned by $w:=[\Gr^{v,(1)}_{d,\ve}] - [\Gr^{v,(2)}_{d,\ve}]$ 
in $\H_{irr}(\Gr^v_{d,\ve})$
is a representation of the algebra generated by $\t_{\ve,d}$ 
of scalar $-d$, i.e., $\t_{\ve,d} . w=-dw$. 
If $a=2$ (resp. $a=-2$), then we have two isomorphic representations of eigenvalue $d$ (resp. $-d$) for $x$ in two distinct orbits. A contradiction. 
Therefore the only value of $a$ is $0$. The statement is thus proved when $\ve=1$. 

Assume now $\ve=-1$.  We have a natural embedding $\P_{\ve-2}(v) \to \P_{\ve}(v)$ defined by $\mu \mapsto (2, \mu)$. 
In light of Proposition~\ref{i-homo} and Lemma~\ref{defining-t}, 
we have $\phi_{v,v-2}(\t_{\ve, d}) = - (\t_{\ve, d-1}+b)$ where $b=\pm 1$. 
Applying the above argument for $\ve=1$, we have
\begin{align}
\label{symp-phi}
\phi_{v, v-2} (\t_{\ve,d}) = - (\t_{\ve, d-1}+1).
\end{align}
So $\phi_{v,v-4} (\t_{\ve,d}) = \phi_{v-2,v-4} \phi_{v, v-2} (\t_{\ve,d}) = - ( -(\t_{\ve, d-2}+1)+1)= \t_{\ve, d-2}.$
The statement holds for $\ve=-1$, and the proof is therefore finished.
\end{proof}

\begin{prop}
\label{t-explicit}
The element $\t_{\ve,d}$ admits the following description.
\[
\t_{\ve, d} = (-1)^{d-\delta_{1,\ve}} \left ( [T^*_{Y_{\ve}}(\Gr^v_{d, \ve} \times \Gr^v_{d,\ve}) ] + (-1)^{\delta_{1,\ve}} \delta_{-1, (-1)^d}  [T^*_{\Delta}(\Gr^v_{d, \ve} \times \Gr^v_{d,\ve}) ]\right ),
\]
where $Y_{\ve}=\{ (F, F') \in \Gr^v_{d,\ve}\times \Gr^v_{d,\ve} | | F\cap F'|=d-1\}$.
\end{prop}

\begin{proof}
By~\cite[Thm. 2.7.26]{CG97}, we have
\[
\e_{\ve} \f_{\ve} = (-1)^{1-\delta_{1,\ve}} [T^*_{Y_{\ve}}( \Gr^2_{1,\ve}\times \Gr^{2}_{1,\ve})], \quad \mbox{when}\ d=1. 
\]
Now apply (\ref{symp-phi}) to obtain the result for $\ve=-1$, and the case $\ve=1$, $d$ odd.

When $\ve=1$ and $d=2$, there is
\[
\t_{\ve,d} = - ( [T^*_{Y_{\ve}}(\Gr^4_{2, \ve} \times \Gr^4_{2,\ve}) ] + a),\quad \mbox{for some $a\in \mbb Z$}.
\] 
But by Proposition~\ref{transfer-t}, 
$\phi_{4,0}(\t_{\ve, 2}) = \t_{\ve, 0}=0$ implies that $a=0$. This shows that the statement holds when $\ve=1$ and $d$ even. The proposition is thus proved.
\end{proof}

\subsection{Relations in Schur algebras}
\label{Schur}

With the above analysis on rank one, we have the following result.
Recall the map $\phi_v$ from (\ref{phi}).

\begin{prop}
\label{Schur-def}
The elements $\phi_v(\e_{i,\ve})$, $\phi_v(\h_{i,\ve})$ satisfy the defining relations of $\U(\mathfrak{sl}_n^{\theta})$.
\end{prop}

\begin{proof}
Assume that $n$ is odd. Then the first four defining relations in (\ref{Def-odd}) can be verified as in~\cite{CG97}, except when $i$ or $j$ is $r$ in the fourth relation.
The latter can be checked directly in a similar way as the third one. 
So all defining relations in (\ref{Def-odd}) are known except the case when $(i, j) = (r, r+1), (r+1, r)$, i.e., the fifth relation.
But the latter cases are reduced to (\ref{rank-one-Serre}), and so the statement in the proposition holds for $n$ being odd.

Assume that $n$ is even. Then there is a natural inclusion $Z_{n, v,\ve}\subseteq Z_{n+1,v,\ve}$ by adding an extra $F_r$ in the involved flags $F$. 
In this case all $\phi_v(\e_{i,\ve})$ are restrictions of the corresponding elements from $\H(Z_{n+1,v,\ve})$ except $\phi_v(\e_{r,\ve})$ which is 
the restriction of $\e_{r,\ve}'\e_{r+1,\ve}'-\h_{r,\ve}'$ from $\H(Z_{n+1,v,\ve})$ in light of Proposition~\ref{t-explicit} (see also (\ref{t-d}). 
A standard algebraic argument shows that the  relations (\ref{Def-even}) are consequences of the relations in $\H(Z_{n+1, v,\ve})$ corresponding to (\ref{Def-odd}).
For the convenience of the reader, let us produce the proof here. To ease  the burden of notations, we  write $\e_{i,\ve}$ for $\phi_v(\e_{i,\ve})$ in this proof.
The first four relations in (\ref{Def-even})  can be checked in a way similar to that of (\ref{Def-odd}). It is reduced to show the last relation, that is  
\begin{align}
&\e_{j,\ve} \e_{j,\ve} \e_{i,\ve} - 2 \e_{j,\ve} \e_{i,\ve} \e_{j,\ve} + \e_{i,\ve} \e_{j,\ve} \e_{j,\ve} =0; \label{symp-even-a}\\
&\e_{i,\ve} \e_{i,\ve} \e_{j,\ve} - 2 \e_{i,\ve} \e_{j,\ve} \e_{i,\ve} + \e_{j,\ve} \e_{i,\ve} \e_{i,\ve} =  \e_{j,\ve}; \quad \mbox{if} \ c_{ij}= -1, i=\theta(i). \label{symp-even-b}
\end{align}

The case when $c_{ij}=-1$ and $i=\theta(i)$ is  $i=r$ and $j= r \pm 1$. 
Let $\bj_{r,\ve}=\sum_A [Z_A]$ where $A$ runs over all diagonal matrices in $\Theta_v$ such that $a_{r+1,r+1}=0$. This is an  idempotent in $\H(Z_{n+1v,\ve})$ and
$\H(Z_{n,v,\ve}) \cong \bj_{r,\ve} \H(Z_{n+1,v,\ve})\bj_{r,\ve}$. We shall identify $\H(Z_{n,v,\ve})$ as a subalgebra in $\H(Z_{n+1,v,\ve})$ under the isomorphism. 
It is clear that if $j\neq \theta (j)$ then the image of $\e_{j,\ve} \in \H(Z_{n, v,\ve})$ under the identification
is $ \e'_{j,\ve} \bj_{r,\ve}$ if $j< r$ or $\e'_{j+1,\ve} \bj_{r,\ve}$ if $j>r$. 
Moreover, $\e_{r,\ve} = (\e_{r,\ve}'\e_{r+1,\ve}'-\h_{r,\ve}') \bj_{r,\ve}$ under the identification. 
For (\ref{symp-even-a}), it is enough to show that the image, say $S_{ji}$, of the left-hand side of (\ref{symp-even-a}) is zero. By the relations  satisfied in $\H(Z_{n+1, v,\ve})$, we have, for $j<i$, 
\begin{align*}
S_{ji} & =-  \e'_{j,\ve}  \e'_{j,\ve}  \h'_{i,\ve}  \bj_{i,\ve}  
-  \h'_{i,\ve} \e'_{j,\ve}  \e'_{j,\ve} \bj_{i,\ve} 
+ 2  \e'_{j,\ve}  \h'_{i,\ve}  \e'_{j,\ve}  \bj_{i,\ve} \\
&= - \e'_{j,\ve} [ \e'_{j,\ve},  \h'_{i,\ve}]  \bj_{i,\ve} - [ \h'_{i,\ve},  \e'_{j,\ve}]  \e'_{j,\ve} 
 \bj_{i,\ve} \\
& = - \e'_{j,\ve} \e'_{j,\ve}  \bj_{i,\ve} +  \e'_{j,\ve}  \e'_{j,\ve}  \bj_{i,\ve}=0.
\end{align*}
For $j<i$, the proof is similar. 
Thus the equality in (\ref{symp-even-a}) holds. 

The equality in (\ref{symp-even-b}) can be verified in a similar way, though more complicated.
More precisely, we want to show that the image of the equality holds in $\H(Z_{n+1, v,\ve})$. 
Let $S_{ij}$ be the image of the left-hand side of (\ref{symp-even-b}) under the above identificaton. 
Then it is enough to show that $S_{ij} =  \e'_{j,\ve}  \bj_{r,\ve}$ if $j<i$ or $\e'_{j+1,\ve}\bj_{r,ve}$ if $j<i$. 
Assume that $j<i$. Since $  \e'_{r,\ve}  \bj_{r,\ve} =0$, 
we see that the first term in $S_{ij}$ can be simplified to  
\begin{align*}
\left (\frac{1}{2} (\e'_{i,\ve})^2 (\e'_{i+1,\ve})^2 + 2  \e'_{i,\ve}  \e'_{i+1,\ve} - \h'_{i,\ve}  \e'_{i,\ve}  \e'_{i+1,\ve} -  \e'_{i,\ve}  \e'_{i+1,\ve}  \h'_{i,\ve}  + ( \h'_{i,\ve})^2 \right ) 
 \e'_{j,\ve}  \bj_{i,\ve}.
\end{align*} 
Similarly, the third term in $S_{ij}$ can be simplified to 
\begin{align*}
 \e'_{j,\ve} \left (\frac{1}{2} ( \e'_{i,\ve})^2 ( \e'_{i+1,\ve})^2 
+ 2  \e'_{i,\ve}  \e'_{i+1,\ve} -  \h'_{i,\ve}  \e'_{i,\ve}  \e'_{i+1,\ve} 
- \e'_{i,\ve}  \e'_{i+1,\ve}  \h'_{i,\ve} + ( \h'_{i,\ve})^2 \right )  \bj_{i,\ve}.
\end{align*}
The second term in $S_{ij}$ can be simplified as follows.
\begin{align*}
\left (\frac{1}{2} \e'_{i,\ve}  \e'_{j,\ve}  \e'_{i,\ve} (\e'_{i+1,\ve})^2 + 2  \e'_{i,\ve} 
 \e'_{j,\ve}  \e'_{i+1,\ve} -  \h'_{i,\ve}  \e'_{j,\ve}  \e'_{i,\ve}  \e'_{i+1,\ve} 
-  \e'_{i,\ve}  \e'_{i+1,\ve}  \e'_{j,\ve}  \h'_{i,\ve} 
+  \h'_{i,\ve}  \e'_{j,\ve}  \h'_{i,\ve} \right )  \bj_{i,\ve}.
\end{align*}
With these and the relations satisfied in $\H(Z_{n+1, v,\ve})$, we can see that $S_{ij}=  \e'_{j,\ve}  \bj_{i,\ve}$ if $j<i$. 
Here we skip some details. The case for $j>i$ can be shown similarly. The equality in (\ref{symp-even-b}) is thus proved. 
Hence the proposition holds. 
\end{proof}

We are ready to give a proof of Theorems~\ref{U-iso-odd} and~\ref{U-iso-even}.

\begin{proof}[Proof of Theorems~\ref{U-iso-odd} and~\ref{U-iso-even}]
Due to Proposition~\ref{Schur-def}, we know that the map $\U(\mathfrak{sl}^{\theta}_n) \to \bU_{n,\ve}$ is a surjective algebra homomorphism.
Due to Proposition~\ref{H-estimate}, the homomorphism is an  isomorphism. Theorems~\ref{U-iso-odd} and~\ref{U-iso-even} are proved. 
\end{proof}

Finally, we discuss 
the case $(n, v, \ve)= ( even, odd, 1)$. In this case, 
the varieties $\F_{n,v,\ve}$ and $Z_{n,v,\ve}$ defined previously are empty and, to rectify this defect, we substitute them by 
\begin{align}
\label{rectify}
\begin{split}
\F_{n,  v,\ve} & :=\{ F\in \F_{n+1,v,\ve}| F_r \ \mbox{is maximal isotropic}\}, \\
Z_{n, v,\ve}& : = \{(x, F, F')\in Z_{n+1, v, \ve}| F, F' \in \F_{n, v, \ve}\}, 
\quad \mbox{when} \ (n, v, \ve)= ( even, odd, 1).
\end{split}
\end{align}

Let us put a superscript $'$ on the generators in $\H_{n+1,\ve}$ as $\e'_{i,\ve}, \f'_{i,\ve}, \h'_{i,\ve}$.
Consider their images $\phi_v (\e'_{i,\ve}), \phi_v (\f'_{i,\ve}), \phi_v(\h'_{i,\ve})$ for $1\leq i\leq r-1$ and $ \phi_{v} (\e'_{r,\ve} \e'_{r+1,\ve} -\h'_{r,\ve})$
in $\H (Z_{n+1,v,\ve})$. 
Consider the following elements in $\H(Z_{n, v,\ve})$ as the restrictions of the above to $Z_{n, v,\ve}$.
\begin{align*}
\begin{split}
\h_{i,\ve; v} &= 
\begin{cases}
\phi_v(\h'_{i,\ve})|_{Z_{n, v, \ve}}, & 1\leq i\leq r-1,\\
- \phi_v(\h'_{n-i,\ve})|_{Z_{n, v, \ve}}, & 1\leq n-i \leq r-1,\\
0, & i=r,
\end{cases}\\
\e_{i, \ve; v}&=
\begin{cases}
  \phi_v (\e'_{i,\ve})|_{Z_{n, v,\ve}}, & 1\leq i\leq r-1,\\
\phi_v (\f'_{n-i,\ve})|_{Z_{n, v,\ve}}, & 1\leq  n-i \leq r-1,\\
\phi_{v} (\e'_{r,\ve} \e'_{r+1,\ve} -\h'_{r,\ve})|_{Z_{n, v,\ve}}, & i=r.
\end{cases}
\end{split}
\end{align*}
By applying the same argument in the proof of Proposition~\ref{Schur-def} for $n$ even, we have 

\begin{prop}
\label{Schur-def-2}
The assignments $h_{i,\theta} \to \h_{i, \ve;v}$ and  $e_{i,\theta} \mapsto \e_{i,\ve;v}$ for various $i$ define a surjective algebra homomorphism
$\U(\mathfrak{sl}_n^{\theta}) \to \H(Z_{n, v,\ve})$.
\end{prop}

We end this section with a remark. 

\begin{rem}
(1). By Propositions~\ref{H-estimate} and~\ref{Schur-def-2}, $\H(Z_{n, v, \ve})$ is generated by $\phi_v(\e_{i,\ve})$ and  $\phi_v(\h_{i,\ve})$ for various $i$.
It is isomorphic to the hyperoctahedral Schur algebras studied in~\cite{G97}.
See~\cite[Conjecture 5.3.4]{Li19a} for  a more general conjecture on $\sigma$-quiver varieties. 

(2). Propositions~\ref{Schur-def} and ~\ref{Schur-def-2} can be shown by exploiting the  equivariant (K-)homology  techniques in~\cite{V93}, 
which will appear else where.
For  the case of type $B/C$ and $n$=odd, see also the arXiv preprints
\href{https://arxiv.org/abs/1911.00851}{arXiv:1911.00851}
and 
\href{https://arxiv.org/abs/1911.07043}{arXiv:1911.07043}  by Ma et al.

(3). The stabilization  in Section~\ref{Z} does not apply to the newly defined $Z_{n,v,\ve}$ in (\ref{rectify}) due to the following inconsistency: 
for a triple $(x, F, F')\in Z_{n, v,\ve}$, the largest part in the Jordan type of $x$ could be $n+1$ while the dimension jump for preserving the maximal isotropic property of $F$ and $F'$ needs to be an even multiple of $n$. 

(4). A presentation of $\H(Z_{n, v,\ve})$ is given in~\cite{LZ19}. It is desirable to see if one can check geometrically, or find geometric meanings of,  the extra relations other than those for
$\U(\mathfrak{sl}_n)$.
See~\ref{t-defining} for the $n=2$ case.

(5). Thanks to the analysis in Section~\ref{variant}, Theorems~\ref{U-iso-odd} and~\ref{U-iso-even} 
for the cases $(v, \ve) =(2\ell, 1)$ can be deduced from similar results in the case $(v,\ve)=(2\ell -1, 1)$. 
This method avoids the unpleasant extra treatment caused by the disconnectedness of the group $\G_{v,\ve}$. 
\end{rem}

\subsection{A natural representation}

In this section, we shall lift the action in (\ref{rank-one-action}) to general cases. 
Recall $\F_{n,v,\ve}$ is the variety of $n$-step isotropic flags in $V$ and 
$\Lambda_v$ from (\ref{Lambda-v}).
For convenience, we set
$
\F_{\O,\ve}=\O, \ [\F_{\O,\ve}]=0.
$
Let $W_{n, v,\ve}$ be the space spanned by the fundamental classes $[\F_{\mbf d,\ve}]$ for $\mbf d\in \Lambda_{v}$. 
Note that $W_{3,v,\ve}=W_{v,\ve}$ in Section~\ref{n=3}.
To each $i$, set
\begin{align*}
\mbf d_{i,\ve}^+ & =
\begin{cases}
\mbf d + \delta_i + \delta_{n+1-i} - \delta_{i+1} - \delta_{n-i}  & \mbox{if}\ d_{i+1}\geq 1,\\
\O &\mbox{o.w.}
\end{cases}\\
\mbf d_{i,\ve}^- & =
\begin{cases}
\mbf d - \delta_i - \delta_{n+1-i} + \delta_{i+1} + \delta_{n-i}  & \mbox{if}\ d_{i}\geq 1,\\
\O &\mbox{o.w.}
\end{cases}
\end{align*}

\begin{prop}
\label{n-W}
The space $W_{n,v,\ve}$ admits an irreducible  $\U(\mathfrak{sl}^{\theta}_n)$-module structure defined by 
\begin{align*}
\h_{i,\ve} & * [\F_{\mbf d,\ve}]  = ( d_{i} - d_{i+1} )[\F_{\mbf d, \ve}], \\
\e_{i,\ve}&  * [\F_{\mbf d,\ve}]  = 
(\mbf d^+_{i, \ve})_i 
[\F_{\mbf d_{i,\ve}^+, \ve}],\\
\f_{i,\ve} & * [\F_{\mbf d,\ve}]=
(\mbf d^-_{i,\ve})_{i+1} [\F_{\mbf d_{i,\ve}^-, \ve}],\quad \forall 1\leq i\leq n-1.
\end{align*}
\end{prop}

\begin{proof}
Assume that $n$ is odd. 
Let $\U(\mathfrak{sl}^{\theta}_n)$ act on $W_{n,v,\ve}$ via  
$\phi_v(\e_{i,\ve}) $, $\phi_v(\f_{i,\ve})$ and $\phi_v(\h_{i,\ve}) $. 
We only need to check that the structure constants of the actions are of the desired ones.

For any $\mbf d\in \Lambda_{v}$, we define
\begin{align*}
Y^+_{i,\mbf d,\ve}  & =\{ (F, F')\in \F_{\mbf d^+_{i,\ve}} \times \F_{\mbf d,\ve} | F_i \overset{1}{\supset} F'_i, F_j= F'_j, \forall j\neq i, n-i\}, \ \mbox{if} \ d_{i+1} \geq 1+\delta_{i,r}.\\
Y^-_{i,\mbf d,\ve} & =\{ (F, F')\in \F_{\mbf d^-_{i,\ve}} \times \F_{\mbf d,\ve} | F_i\overset{1}{\subset} F'_i, F_j=F'_j,\forall j\neq i, n-i\},\ \mbox{if} \ d_{i}\geq 1.
\end{align*}
Note that
$
Y^-_{i,\mbf d,\ve} = Y^+_{\theta(i), \mbf d,\ve}, \ \mbox{where}\ \theta(i) = n-i.
$
For all $1\leq i\leq n-1$, we have
\begin{align*}
\phi_v(\e_{i,\ve}) &= \sum_{\mbf d\in \Lambda_{n,v}: d_{i+1} \geq 1} 
(-1)^{ d_i+\delta_{i,r+1} (\delta_{-1,(-1)^v}  + 1 -\delta_{1,\ve})}
[T^*_{Y^+_{i,\mbf d,\ve}} ( \F_{\mbf d^{+}_{i,\ve},\ve}\times \F_{\mbf d,\ve})],\\
\phi_v(\f_{i,\ve}) & = \sum_{\mbf d\in \Lambda_{n,v}: d_{i+1} \geq 1}  (-1)^{d_{i+1} +\delta_{i,r} (\delta_{-1,(-1)^v}  + 1 -\delta_{1,\ve}) }
[T^*_{Y^-_{i,\mbf d,\ve}} ( \F_{\mbf d^{-}_{i,\ve},\ve}\times \F_{\mbf d,\ve})],\\
\phi_v(\h_{i,\ve}) & = \sum_{\mbf d\in \Lambda_{n,v}} ( d_{i} - d_{i+1} )
[T^*_{\Delta}(\F_{\mbf d,\ve}\times \F_{\mbf d,\ve})].
\end{align*}
Now applying~\cite[Thm. 2.7.26]{CG97} and~\cite[Lemma 8.5]{N98}, we obtain the desired results.
In particular, the nontrivial case is the third formula when $i=r$. In this case, we have from the rank one calculus that 
\begin{align*}
\f_{r,\ve}  * [\F_{\mbf d,\ve}] 
=\chi ( \Gr_{1,\ve}^{d_{r+1}+2} ) [\F_{\mbf d,\ve}]
= (d_{r+1} +2 - \delta_{1, \ve} \delta_{-1, (-1)^{v}} ) [\F_{\mbf d^+_{r+1,\ve}, \ve}]
=(\mbf d^+_{r+1, \ve})_{r+1} 
[\F_{\mbf d_{r+1,\ve}^+, \ve}].
\end{align*}
The proposition is thus proved for $n$ odd.

The $n$-being-even case is due to the $n$-being-odd case and  (\ref{t-act}). The proof is finished.
\end{proof}


Recall the algebra $\mathfrak{sl}_n^{\theta}=\langle e_{i,\theta}, h_{i,\theta}\rangle$ and the involution $\tau$ from Section~\ref{Pre}.
Let $\mbf d^0\in \Lambda_v$ such that $\mbf d^0_i= 2d\delta_{i, r+1}$.
By using Proposition~\ref{n-W}, a quick computation  yields via $\phi_v$: for all $1\leq i\leq r$
\begin{align*}
\begin{split}
\tau(h_{i,\theta})* [\F_{\mbf d^0,\ve}] = 2d \delta_{i, r} [\F_{\mbf d^0,\ve}],\
\tau(h'_{i,\theta})* [ \F_{\mbf d^0,\ve}] = 2d \delta_{i, r} [\F_{\mbf d^0,\ve}],\
\tau(e_{i,\theta}) * [\F_{\mbf d^0,\ve}] =0.
\end{split}
\end{align*}
Let ${}^{\tau}W_{n, v,\ve}$ denote the $\mathfrak{sl}_n^{\theta}$-module obtained from $W_{n, v,\ve}$ twisted by $\tau$ in (\ref{tau}). 
We see that for $n$ odd,  ${}^{\tau}W_{n, v,\ve}$ is the finite dimensional irreducible representation $L'(\omega; \omega')$ 
with $\omega=\omega'=(0,\cdots, 0, 2d)$. 
In the same manner, one can check that $W_{n, v,\ve} \cong L'( (d,0,\cdots,0), (2d, 0,\cdots,0))$.

\begin{rem}
Note that the involution $(x, F, F') \mapsto (x, F', F)$ on $Z_{n, v,\ve}$ defines an anti-involution $\tau$ on $\H_*(Z_{n, v,\ve})$ such that
$\tau(\e_{i,\ve}) = \f_{i,\ve}$ and $\tau(\h_{i,\ve}) =\h_{i,\ve}$.
Hence this anti-involution is not the counterpart of (\ref{tau}).  
\end{rem}


\begin{thebibliography}{99999}\frenchspacing
 
 
\bibitem[AM18]{AM18}
T.  Arakawa and A. Moreau,  
{\em On the irreducibility of associated varieties of W-algebras}, J. Algebra {\bf 500} (2018), 542-568. 


\bibitem[BKLW]{BKLW}
H. Bao, J. Kujawa, Y. Li and W. Wang, 
{\em Geometric Schur duality of classical type}, Transform. Groups {\bf 23} (2018), no. 2, 329-389. 
 
\bibitem[BLM]{BLM} 
A. Beilinson, G. Lusztig and R. McPherson,
          {\em A geometric setting for the quantum deformation of $GL_n$}, 
Duke Math. J., {\bf 61} (1990), 655-677.
 
\bibitem[BNPP]{BNPP}
C.  Bendel, D. Nakano, B. Parshall and C. Pillen, 
{\em Cohomology for quantum groups via the geometry of the nullcone}. 
Mem. Amer. Math. Soc. {\bf 229} (2014), no. 1077.

 \bibitem[BG99]{BG99}
A. Braverman and D.  Gaitsgory,
{\em On Ginzburg's Lagrangian construction of representations of $GL(n)$},  Math. Res. Lett. 6 (1999), no. 2, 195-201. 

 
 \bibitem[CG97]{CG97}
N. Chriss and V. Ginzburg,
{\em Representation theory and complex geometry},
Birkh$\ddot{\mbox{a}}$user Boston, Inc., Boston, MA, 1997.
 
\bibitem[CM93]{CM93}
D.H. Collingwood and W.M. McGovern,
{\em Nilpotent orbits in semisimple Lie algebra: an introduction}
Van Nostrand Reinhold Mathematics Series. Van Nostrand Reinhold Co., New York, 1993. 
  
\bibitem[FLLLW]{FLW}
Z. Fan, C.-J. Lai, Y. Li,  L. Luo and W. Wang,
{\em Affine flag varieties and quantum symmetric pairs},   Mem. Amer. Math. Soc. {\bf 265} (2020), no. 1285, v+123 pp.
Available at arXiv:1602.04383.

\bibitem[FL19]{FL19}
Z. Fan and Y. Li,
{\em Positivity of canonical bases under comultiplication}, 
 Int. Math. Res. Not. IMRN, \href{https://doi.org/10.1093/imrn/rnz047}{\color{blue}{https://doi.org/10.1093/imrn/rnz047}}. 


\bibitem[FJLS]{FJLS15}
B. Fu, D. Juteau, P. Levy and E. Sommers,
{\em Generic singularities of nilpotent orbit closures}, 
Adv. Math. {\bf 305} (2017), 1-77. 

\bibitem[G91]{G91}
V. Ginzburg, 
{\em Lagrangian construction of the enveloping algebra $\U(sl_n)$}. C. R. Acad. Sci. Paris S\'{e}r. I Math. 312 (1991), no. {\bf 12}, 907-912. 

\bibitem[G98]{G98}
V. Ginzburg, 
{\em Geometric methods in the representation theory of Hecke algebras and quantum groups.} 
Notes by Vladimir Baranovsky, NATO Adv. Sci. Inst. Ser. C Math. Phys. Sci., {\bf 514}, {\em Representation theories and algebraic geometry} (Montreal, PQ, 1997), 127-183, Kluwer Acad. Publ., Dordrecht, 1998. 


\bibitem[Gr97]{G97} {R. Green},
     \emph{Hyperoctahedral {S}chur algebras},
   J. Algebra   \textbf{192} (1997), no. 1, 418-438.


\bibitem[KP82]{KP82}
 H. Kraft and C. Procesi,
 {\em On the geometry of conjugacy classes in classical groups}, Comment. Math. Helv. {\bf 57} (1982), no. 4, 539-602.

\bibitem[Ku02]{K02}
S.  Kumar, 
{\em Kac-Moody groups, their flag varieties and representation theory,} 
Progress in Mathematics, {\bf 204}. Birkh\"{a}user Boston, Inc., Boston, MA, 2002.

\bibitem[L19a]{Li19a} Y. Li, 
{\em Quiver varieties and symmetric pairs}, 
Represent. Theory, {\bf 23} (2019), 1-56.

\bibitem[L19b]{Li19b} Y. Li,
{\em Spaltenstein varieties of pure dimension}, Proceedings of the AMS, {\bf 148}, (2020) 133-144. 




\bibitem[LW18]{LW18} 
Y. Li and W. Wang,
{\it Positivity vs negativity of canonical bases},  
Bull. Inst. Math. Acad. Sin. (N.S.) {\bf 13} (2018), no. 2, 143-198.


\bibitem[LZ19]{LZ19}
Y. Li and J. Zhu,
{\it Quasi-split symmetric pairs of $\U(\mathfrak{gl}_n)$ and their Schur algebras}, 
Nagoya Mathematical Journal, 1-27. 
\href{doi:10.1017/nmj.2020.16}{doi:10.1017/nmj.2020.16}


\bibitem[Lu88]{Lu88}
G. Lusztig,
{\it Cuspidal local systems and graded Hecke algebras,} I,
 Inst. Hautes \'{E}tudes Sci. Publ. Math.  {\bf 67} (1988), 145-202. 

\bibitem[Lu07]{Lu07} 
G. Lusztig, 
{\it A class of perverse sheaves on a partial flag manifold}, Represent. Theory {\bf 11} (2007), 122-171.

\bibitem[N98]{N98}  
H. Nakajima, 
{\em Quiver varieties and Kac-Moody algebras},  
Duke Math. J. {\bf 91} (1998), no. 3, 515-560. 


\bibitem[SS99]{SS99} 
M. Sakamoto and T. Shoji,
\emph{Schur-Weyl reciprocity for Ariki-Koike algebras},
{J. Algebra} \textbf{221} (1999), no. 1, 293-314.

\bibitem[V93]{V93}
\'{E}. Vasserot,
{\em Repr\'{e}sentations de groupes quantiques et permutations},
Ann. Sci. \'{E}cole Norm. Sup. (4) {\bf 26} (1993), no. 6, 747--773. 

\bibitem[VV03]{VV03}
M. Varagnolo and \'{E}. Vasserot,
{\em Perverse sheaves and quantum Grothendieck rings},
in {\em Studies in Memory of Issai Schur}, Progress in Mathematics,  {\bf 210} (2003)  Birkh\"{a}user, Boston, MA,  345-365.


\bibitem[W18]{W18}
H. Watanabe,
{\em Crystal basis theory for a quantum symmetric pair $(\mbf U,\mbf U^{\jmath})$,}
arXiv:1704.01277.
\end{thebibliography}
\end{document}